\documentclass{amsart}
\usepackage{amsmath}
\usepackage{amsfonts}
\usepackage{amssymb}
\usepackage{mathrsfs}
\usepackage{amscd}
\usepackage{url}

\newcommand{\mt}[1]{\mathtt{#1}}

\newcommand{\re}{\mathbb{R}}
\newcommand{\cpx}{\mathbb{C}}

\newcommand{\N}{\mathbb{N}}

\newcommand{\half}{\frac{1}{2}}
\newcommand{\lmd}{\lambda}

\newcommand{\eps}{\epsilon}

\newcommand{\dt}{\delta}

\def\af{\alpha}

\def\rank{\mbox{rank}}

\newcommand{\sig}{\sigma}
\newcommand{\Sig}{\Sigma}

\newcommand{\reff}[1]{(\ref{#1})}

\newcommand{\mc}[1]{\mathcal{#1}}

\newcommand{\supp}[1]{\mbox{supp}(#1)}

\newcommand{\bdes}{\begin{description}}
\newcommand{\edes}{\end{description}}

\newcommand{\bal}{\begin{align}}
\newcommand{\eal}{\end{align}}

\newcommand{\bnum}{\begin{enumerate}}
\newcommand{\enum}{\end{enumerate}}

\newcommand{\bit}{\begin{itemize}}
\newcommand{\eit}{\end{itemize}}

\newcommand{\bea}{\begin{eqnarray}}
\newcommand{\eea}{\end{eqnarray}}
\newcommand{\be}{\begin{equation}}
\newcommand{\ee}{\end{equation}}

\newcommand{\baray}{\begin{array}}
\newcommand{\earay}{\end{array}}

\newcommand{\bsry}{\begin{subarray}}
\newcommand{\esry}{\end{subarray}}

\newcommand{\bca}{\begin{cases}}
\newcommand{\eca}{\end{cases}}

\newcommand{\bcen}{\begin{center}}
\newcommand{\ecen}{\end{center}}

\newcommand{\bbm}{\begin{bmatrix}}
\newcommand{\ebm}{\end{bmatrix}}

\newcommand{\bmx}{\begin{matrix}}
\newcommand{\emx}{\end{matrix}}

\newcommand{\bpm}{\begin{pmatrix}}
\newcommand{\epm}{\end{pmatrix}}

\newcommand{\btab}{\begin{tabular}}
\newcommand{\etab}{\end{tabular}}

\newtheorem{theorem}{Theorem}[section]
\newtheorem{thm}[theorem]{Theorem}

\newtheorem{prop}[theorem]{Proposition}
\newtheorem{lem}[theorem]{Lemma}
\theoremstyle{definition}
\newtheorem{exm}[theorem]{Example}
\newtheorem{rmk}[theorem]{Remark}
\begin{document}

\title[The Hierarchy of Local Minimums]
{The Hierarchy of Local Minimums in Polynomial Optimization}

\author[Jiawang Nie]{Jiawang Nie}
\address{Department of Mathematics, University of California San Diego,
9500 Gilman Drive,  La Jolla, CA 92093, USA.}

\email{njw@math.ucsd.edu}
\thanks{The research was partially supported by the NSF grants DMS-0844775
and DMS-1417985.}

\begin{abstract}
This paper studies the hierarchy of local minimums of a polynomial
in the vector space $\re^n$.
For this purpose, we first compute $H$-minimums,
for which the first and second order
necessary optimality conditions are satisfied.
To compute each $H$-minimum,
we construct a sequence of semidefinite relaxations,
based on optimality conditions.
We prove that each constructed sequence
has finite convergence, under some generic conditions.
A procedure for computing all local minimums is given.
When there are equality constraints, we have similar results for
computing the hierarchy of critical values
and the hierarchy of local minimums.
\end{abstract}

\keywords{critical point, local minimum, optimality condition,
polynomial optimization, semidefinite relaxation, sum of squares}

\subjclass{65K05, 90C22, 90C26}

\maketitle

\section{Introduction}

Let $f$ be a polynomial in $x:=(x_1,\ldots,x_n)$.
A point $u\in \re^n$ (the space of $n$-dimensional real vectors)
is a local minimizer of $f$
if there exists $\eps>0$ such that
$f(u)$ is the smallest value of $f$
on the ball $B(u,\eps):=\{x \in \re^n \mid  \|x-u\| \leq \eps\}$.
(Here $\| \cdot \|$ is the standard Euclidean norm on $\re^n$.)
Such $f(u)$ is called a local minimum of $f$.
The set of all local minimums of $f$ is always finite, if it is not empty.
This is because $f$ achieves only finitely
many values on the points where its gradient vanishes.
(This fact can be implied by Lemma~3.2 of \cite{Nie-jac},
or by the proof of Theorem~8 of \cite{NDS}.)
We order the local minimums monotonically
as $\nu_1 < \nu_2 < \cdots < \nu_N$.
We call $\nu_r$ the $r$-th local minimum of $f$.
The sequence $\{ \nu_r \}$ is called the hierarchy of local minimums of $f$.
This paper studies how to compute this hierarchy.

\medskip
\noindent
{\bf Background}\,
Clearly, if $u$ is a local minimizer of $f$,
then the gradient $\nabla f(u) = 0$ and
the Hessian $\nabla^2 f(u) \succeq 0$ (positive semidefinite).
Conversely, if $\nabla f(u) = 0$ and $\nabla^2 f(u) \succ 0$ (positive definite),
then $u$ is a strict local minimizer (cf.~\cite{Brks}).
This is a basic fact in nonlinear programming.
If $\nabla f(u) = 0$, regardless of $\nabla^2 f(u) \succeq 0$,
$u$ is called a critical point (or stationary point in the literature).
When $u$ is a critical point, if
$\nabla^2 f(u) \succeq 0$ but is singular,
we cannot conclude that $u$ is a local minimizer.
For such cases, higher order derivatives are required to make a judgement.
Indeed, it is NP-hard to check
whether a critical point is a local minimizer or not (cf.~\cite{MurKab}).
To see this, consider the special case that $f$ is a quartic form
(i.e., a homogeneous polynomial of degree four).
The origin $0$ is always a critical point of such $f$.
However, $0$ is a local minimizer of $f$ if and only if
$f$ is nonnegative everywhere.
So, checking local optimality at $0$ is equivalent to
checking whether $f$ is nonnegative everywhere.
As is known, checking nonnegativity of quartic forms is NP-hard
(cf.~\cite{Las01,Nes02}).

However, for generical cases, checking local optimality is easy.
For a degree $d$, let $\re[x]_d$ denote the set of all
real polynomials in $x$ with degrees $\leq d$.
As shown in \cite[Theorem~1.2]{Nie-opcd},
there exists a subset $Z$ of $\re[x]_d$, whose Lebsuge measure is zero,
such that for all $f \in \re[x]_d\backslash Z$,
a point $u$ is a local minimizer of $f$ if and only if $\nabla f(u) =0$ and
$\nabla^2 f(u) \succ 0$.
Note that $\re[x]_d\backslash Z$ is open dense in $\re[x]_d$.
In other words, for generic $f$,
the conditions $\nabla f(u) =0$ and $\nabla^2 f(u) \succ 0$
are sufficient and necessary for $u$ to be a local minimizer.
This fact has been observed quite a lot in practice.

We would like to remark that a polynomial might have no local
minimizers, even if it is bounded from below.
For instance, the polynomial $x_1^2 + (x_1x_2-1)^2$
is bounded from below by zero,
but does not have any local minimizer.
This is because $0$ is the unique critical point
but the Hessian is indefinite at it.
Its infimum over $\re^2$ is zero, but it is not achievable.

When the smallest local minimum $\nu_1$ equals the infimum of $f$ over $\re^n$,
i.e., $f$ has a global minimizer, there exists much work on computing $\nu_1$,
e.g., Lasserre \cite{Las01}, Parrilo \cite{ParMP},
Parrilo and Sturmfels \cite{PS03}.
These methods are based on sum of squares (SOS) relaxations,
and they often get $\nu_1$. When the SOS relaxations
are not exact, the value $\nu_1$ cannot be found by them.
For such cases, the gradient SOS relaxation
method by the author, Demmel and Sturmfels \cite{NDS} is useful.
A major advantage of this method is that we can always get $\nu_1$
if a global minimizer exists (cf.~\cite{NDS,Nie-jac}).
If $f$ does not have a global minimizer, this method
might get a value that is not $\nu_1$. For instance,
when this method is applied to
$x_1^2 + (x_1x_2-1)^2$, we get the value $1$,
which is not a local minimum.
Recently, there is much work on polynomial optimization.
We refer to Lasserre~\cite{LasBok}, Laurent~\cite{Lau},
Marshall~\cite{MarBk} and Scheiderer~\cite{Sch09}.

\medskip
\noindent
{\bf Motivations} \,
In the classical literature of polynomial optimization, most work is
for computing global minimums. However, little is done for
local minimums. If $f$ is unbounded from below,
how can we determine whether it has a local minimizer or not?
If it has one, how can we compute the smallest local minimum $\nu_1$?
If it does not, how can we get a certificate for its nonexistence?
For $k>1$, if $\nu_k$ exists, how can we compute it?
If $\nu_k$ does not exist, how can we identify its nonexistence?
Similar questions can be asked for critical values and
for constrained polynomial optimization.
Such questions are theoretically interesting and mathematically meaningful.
To the best of the author's knowledge, they are mostly open.
Critical values and local minimums have broad applications.
Here, we list some of them.

Tensor eigenvalues are of major interests in multilinear algebra.
Each symmetric tensor $\mc{A}$, of order $m$, can be equivalently represented
by a homogeneous polynomial $\mc{A}x^m$, of degree $m$ (cf.~\cite{CDN14}).
The Z-eigenvalues of $\mc{A}$ are the critical values of $\mc{A}x^m$
over the unit sphere. The Laplacian matrix of a graph
has useful properties (cf.~\cite{Moh92}). Its eigenvalues
can be used to estimate edge-densities in cuts. Similarly,
for hypergraphs, Z-eigenvalues of their Laplacian tensors
have important properties (cf.~\cite{HQ12,LQY13}).
For instance, the second largest Z-eigenvalue can be used to
estimate the bipartition widths.
Moreover, tensor eigenvalues have applications
in signal processing and diffusion tensor imaging
(cf.~\cite{Qi03,Qi10})).
Recently, Cui, Dai and the author \cite{CDN14}
worked on real eigenvalues of symmetric tensors.
The techniques developed in this paper
are very useful for such applications.

Low rank approximations for tensors are important in applications
(cf.~\cite{DeSLim08}).
The main task is to minimize a distance function,
which is a polynomial in a set of variables.
As discovered in \cite{DeSLim08},
such a polynomial often does not have global minimizers.
Hence, local minimizers are mostly wanted in applications.
In practice, people often want to find local minimizers
on which the function values are small.
This amounts to computing the hierarchy of local minimums.
After the hierarchy is obtained,
we are able to judge how good a local minimizer is.

Nonlinear complementarity problems have important applications.
For a function $F:\re^n \to \re^n$, the task is to find
vectors $u\in \re^n$ such that
\be \label{0<u:perp:F(u)>=0}
0 \leq u \perp F(u) \geq 0.
\ee
A broad and interesting class of such function $F$
is the gradient of a polynomial $f$.
Such problems have important applications in
mathematical programs with equilibrium constraints
(cf.~\cite{MPECbook}). For such cases, it can be shown that
$u$ satisfies \reff{0<u:perp:F(u)>=0} if and only if
$u$ is a critical point or local minimizer of $f$
over the nonnegative orthant. By replacing $x$ with
$x^2:=(x_1^2,\ldots,x_n^2)$, the problem is equivalent to
computing the hierarchy of critical points and local minimizers
of $f(x^2)$ in the space $\re^n$.

\medskip
\noindent
{\bf Contributions} \,
This paper studies how to compute the hierarchy of critical values and local minimums,
by using semidefinite relaxations.

First, we study how to get the hierarchy
of local minimums of a polynomial $f$ in the space $\re^n$.
We use the first and second order necessary optimality conditions
\be \label{opcd:1st2nd}
\nabla f(x) = 0, \quad \nabla^2 f(x) \succeq 0
\ee
for constructing semidefinite relaxations.
A point satisfying \reff{opcd:1st2nd} is called an $H$-minimizer,
and such $f(u)$ is called an $H$-minimum.
To compute an $H$-minimum,
we construct a sequence of semidefinite relaxations about moment variables.
We prove that each constructed sequence
has finite convergence to an $H$-minimum, under some generic conditions.
We give a procedure for computing all $H$-minimums.
After they are computed,
we show how to extract the hierarchy of local minimums from them.
The results are shown in Section~3.

\,

Second, we study how to compute the hierarchy of critical values
when there are equality constraints. Like the unconstrained case,
there are also finitely many critical values
and finitely many local minimum values.
To compute each critical value,
we construct a sequence of semidefinite relaxations,
by using optimality conditions.
We show that each constructed sequence has finite convergence.
We give a procedure for computing all real critical values.
Once they are found,
we show how to extract the hierarchy of local minimums from them.
The results are shown in Section~4.

\,

Third, we discuss some extensions and related questions.
This is shown in Section~5.
We begin with a review of necessary backgrounds in Section~2.

\section{Preliminaries}
\setcounter{equation}{0}

\noindent
{\bf Notation}
The symbol $\N$ (resp., $\re$, $\cpx$) denotes the set of
nonnegative integral (resp., real, complex) numbers.
The symbol $\re[x] := \re[x_1,\ldots,x_n]$
denotes the ring of polynomials in $x:=(x_1,\ldots,x_n)$
with real coefficients.
Let $\N^n_d :=\{ \af \in \N^n \mid \af_1+\cdots+\af_n \leq d \}$.
For a polynomial $p$, $\deg(p)$ denotes its degree.
For a real number $t$, $\lceil t \rceil$ denotes
the smallest integer that is greater than or equal to $t$.
For a positive integer $k$, denote $[k]:=\{1,2,\ldots, k\}$.
%
The superscript $^T$ denotes the transpose of a matrix or vector.
For a measure $\mu$, $\supp{\mu}$ denotes its support.

\subsection{Ideals, varieties and real algebra}

Here we review some basic facts in real and complex algebraic geometry.
We refer to \cite{BCR,CLO97} for details.

An ideal $I$ in $\re[x]$ is a subset of $\re[x]$
such that $ I \cdot \re[x] \subseteq I$
and $I+I \subseteq I$. For a tuple $h=(h_1,\ldots,h_m)$ in $\re[x]$,
$\langle h \rangle $ denotes
the smallest ideal containing all $h_i$, which equals the set
$h_1 \cdot \re[x] + \cdots + h_m  \cdot \re[x]$.
The $k$-th {\it truncation} of the ideal $\langle h \rangle $,
denoted as $\langle h \rangle_{k}$, is the set
\[
h_1 \cdot \re[x]_{k-\deg(h_1)} + \cdots + h_m  \cdot \re[x]_{k-\deg(h_m)}.
\]
%
%
A complex variety is the set of
common complex zeros of some polynomials. A real variety is the set of
common real zeros of some polynomials. For a polynomial tuple $h$,
its complex and real varieties are denoted respectively as
\begin{align*}
V_{\cpx}(h) := \{v\in \cpx^n \mid \, h(v) = 0 \}, \quad
V_{\re}(h)  := \{v\in \re^n \mid \,  h(v) = 0 \}.
\end{align*}

A polynomial $\sig$ is said to be sum of squares (SOS)
if $\sig = p_1^2+\cdots+ p_k^2$ for some $p_1,\ldots, p_k \in \re[x]$.
The set of all SOS polynomials in $x$ is denoted as $\Sig[x]$.
For a degree $m$, denote the truncation
$
\Sig[x]_m := \Sig[x] \cap \re[x]_m.
$
For a tuple $g=(g_1,\ldots,g_t)$,
its {\it quadratic module} is the set
\[
Q(g):=  \Sig[x] + g_1 \cdot \Sig[x] + \cdots + g_t \cdot \Sig[x].
\]
The $k$-th truncation of $Q(g)$ is the set
\[
\Sig[x]_{2k} + g_1 \cdot \Sig[x]_{d_1} + \cdots + g_t \cdot \Sig[x]_{d_t},
\]
where each $d_i = 2k - \deg(g_i)$.
For polynomial tuples $h$ and $g$,
the sum $\langle h \rangle + Q(g)$ is called {\it archimedean}
if there exists $p \in \langle h \rangle + Q(g)$ such that
$p(x) \geq 0$ defines a compact set in $\re^n$.
When $\langle h \rangle + Q(g)$ is archimedean,
if a polynomial $f$ is positive on the set $\{ h(x)=0, \, g(x)\geq 0 \}$,
then $f  \in \langle h \rangle + Q(g)$
(cf.~Putinar \cite{Put}).

\subsection{Truncated moment sequences}

Let $\re^{\N_d^n}$ be the space of real sequences indexed by $\af \in \N^n_d$.
A vector in $\re^{\N_d^n}$ is called a
{\it truncated moment sequence} (tms) of degree $d$.
A tms $y \in \re^{\N_d^n}$ defines the Riesz functional $\mathscr{L}_y$
acting on $\re[x]_d$ as
\[
\mathscr{L}_y\Big(\sum_{\af \in \N_d^n}
p_\af x_1^{\af_1}\cdots x_n^{\af_n} \Big) := \sum_{\af \in \N_d^n}  p_\af y_\af.
\]
For convenience, denote $x^\af := x_1^{\af_1}\cdots x_n^{\af_n}$ and
\be \label{df:<p,y>}
\langle p, y \rangle := \mathscr{L}_y(p).
\ee
We say that $y$ admits a representing measure supported in a set $T$
if there exists a Borel measure $\mu$ such that
its support $\supp{\mu} \subseteq T$ and
$y_\af = \int  x^\af \mt{d} \mu$
for all $\af \in \N_d^n$.

Let $q \in \re[x]$ with $\deg(q) \leq 2k$.
The $k$-th {\it localizing matrix} of $q$,
generated by a tms $y \in \re^{\N^n_{2k}}$,
is the symmetric matrix $L_q^{(k)}(y)$ satisfying
\[
vec(p_1)^T \Big( L_q^{(k)}(y) \Big) vec(p_2)  = \mathscr{L}_y(q p_1p_2)
\]
for all $p_1,p_2 \in \re[x]$ with
$\deg(p_1), \deg(p_2) \leq k - \lceil \deg(q)/2 \rceil$.
In the above, $vec(p_i)$ denotes the coefficient vector of the polynomial $p_i$.
When $q = 1$, $L_q^{(k)}(y)$ is called a {\it moment matrix} and is denoted as
\[
M_k(y):= L_{1}^{(k)}(y).
\]
The columns and rows of $L_q^{(k)}(y)$, as well as $M_k(y)$,
are indexed by $\af \in \N^n$ with $2|\af| + \deg(q) \leq 2k$.

Let $g=(g_1,\ldots, g_t)$ be a tuple of polynomials in $\re[x]$, and denote
\[
\mathcal{S}(g) := \{ x \in \re^n \mid g(x) \geq 0 \}.
\]
If a tms $y \in \re^{\N^n_{2k}}$ admits a measure
supported in $\mathcal{S}(g)$, then (cf.~\cite{CuFi05})
\be \label{Lg>=0}
M_k(y)\succeq 0, \quad L_{g_j}^{(k)}(y) \succeq 0 \, (j=1,\ldots, t).
\ee
The reverse is typically not true. Let
$
d_g = \max_{j} \, \{1,  \lceil \deg(g_j)/2 \rceil \}.
$
If $y$ satisfies \reff{Lg>=0} and the rank condition
\be \label{con:fec}
\rank \, M_{k-d_g}(y) \, = \, \rank \, M_k(y),
\ee
then $y$ admits a measure supported in $\mathcal{S}(g)$.
This was shown by Curto and Fialkow \cite{CuFi05}.
When \reff{Lg>=0} and \reff{con:fec} hold,
the tms $y$ admits a unique representing measure $\mu$ on $\re^n$;
moreover, the measure $\mu$ is $r$-atomic with $r=\rank\, M_k(y)$
(i.e., $\supp{\mu}$ consists of $r$ distinct points)
and supported in $\mathcal{S}(g)$.
%
%
The points in $\supp{\mu}$ can be
found by solving some eigenvalue problems
(cf.~Henrion and Lasserre \cite{HenLas05}).
For convenience, we say that $y$ is {\it flat with respect to} $g\geq 0$
if \reff{Lg>=0} and \reff{con:fec} are {\it both} satisfied.

For a tuple $h=(h_1,\ldots, h_m)$ of polynomials,
consider the semialgebraic set
\[
\mathcal{E}(h) := \{ x \in \re^n \mid h(x) = 0 \}.
\]
Clearly, $\mathcal{E}(h) = \mathcal{S}(\, h,-h\, )$.
We say that $y \in \re^{\N^n_{2k}}$
is flat with respect to $ h = 0$
if \reff{con:fec}, where $d_g$ is replaced by
$d_h=\max_i\{1, \lceil \deg(h_i)/2 \rceil\}$, and
\be \label{Lh=0}
L_{h_i}^{(k)}(y) = 0\, (i =1,\ldots, m),
\quad M_k(y) \succeq 0
\ee
are satisfied. If $y$ is flat with respect to $ h = 0$,
then $y$ admits a finitely atomic measure supported in $\mathcal{E}(h)$.

For two tms' $y \in \re^{\N^n_{2k} }$ and $z \in \re^{\N^n_{2l} }$ with $k<l$,
we say that $y$ is a truncation of $z$,
or equivalently, $z$ is an extension of $y$,
if $y_\af = z_\af$ for all $\af \in \N^n_{2k}$.
Denote by $z|_d$ the subvector of $z$
whose entries are indexed by $\af \in \N^n_d$.
Thus, $y$ is a truncation of $z$ if $z|_{2k} = y$.
Throughout the paper, if $z|_{2k} = y$ and $y$ is flat,
we say that $y$ is a flat truncation of $z$.
Similarly, if $z|_{2k} = y$ and $z$ is flat,
we say that $z$ is a flat extension of $y$.
Flat extension and flat truncation are proper criteria
for checking convergence of Lasserre's hierarchy
in polynomial optimization (cf.~\cite{Nie-ft}).

\subsection{Polynomial matrix inequalities}

Let $H \in \re[x]^{\ell \times \ell}$ be a symmetric matrix polynomial,
i.e., $H$ is an $\ell \times \ell$ symmetric matrix and
each entry $H_{ij}$ is a polynomial in $\re[x]$.
The $k$-th localizing matrix of $H$, generated by $y \in \re^{\N^n_{2k} }$,
is the block symmetric matrix $L_{H}^{(k)}(y)$ defined as
{\small
\[
L_{H}^{(k)}(y) \, := \,  \Big( L_{H_{ij}}^{(k)}(y)  \Big)_{1 \leq i, j \leq \ell }.
\]
} \noindent
Each block is a standard localizing matrix.
For $p = (p_1,\ldots, p_{\ell})  \in \re[x]^{\ell}$, denote
\[
vec(p) = [\, vec(p_1)^T \quad   \cdots \quad vec(p_{\ell})^T \,]^T.
\]
Then, one can verify that for all $p \in \re[x]^{\ell}$,
if each $\deg(H_{ij} p_i p_j ) \leq 2k$, then
\[
\mathscr{L}_y(  p^T H p ) =  vec(p)^T \Big( L_{H}^{(k)}(y) \Big) vec(p).
\]
Let $\Sig[x]^{\ell \times \ell}$
be the cone of all possible sums $a_1 a_1^T + \cdots + a_r a_r^T$
with $a_1, \ldots, a_r \in \re[x]^\ell$.
When $\ell =1$, $\Sig[x]^{\ell \times \ell}$ is the cone $\Sig[x]$.
The quadratic module of $H$ is
\[
Q(H):= \Sig[x] + \left\{\mbox{Trace} (H W): \,
W \in \Sig[x]^{\ell \times \ell} \right\}.
\]
The $k$-th truncation of $Q(H)$ is defined as
\[
Q_k(H):= \Sig[x]_{2k} + \left\{ \mbox{Trace} (H W): \,
W \in \Sig[x]^{\ell \times \ell}, \,
\deg(W_{ij} H_{ij}) \leq 2k \,\, \forall \, i,j \right \}.
\]
Consider the semialgebraic set
\[
\mathcal{S}(H) := \{ x \in \re^n \mid \, H(x) \succeq 0 \}.
\]
Let
$
d_H = \max_{i,j} \, \{1, \lceil \deg(H_{ij})/2 \rceil \}.
$
For a tms $y \in \re^{\N^n_{2k}}$, we say that
$y$ is {\it flat with respect to} $H \succeq 0$ if
\be \label{rk:flat:H>=0}
M_k(y) \succeq 0, \quad L_{H}^{(k)}(y) \succeq 0, \quad
\rank \, M_{k-d_H}(y) = \rank\, M_k(y).
\ee

\begin{prop} \label{flat:H>=0}
If a tms $y \in \re^{\N^n_{2k}}$ is flat
with respect to a polynomial matrix inequality $H \succeq 0$,
then it admits a unique representing measure;
moreover, this measure is $\rank\,M_k(y)$-atomic,
and is supported in $\mathcal{S}(H)$.
\end{prop}
\begin{proof}
Because of $M_k(y) \succeq 0$ and the rank condition in \reff{rk:flat:H>=0},
the tms $y$ admits a unique measure $\mu$ on $\re^n$,
and $\mu$ is $r$-atomic, where $r = \rank\,M_k(y)$,
by Theorem~1.1 of Curto and Fialkow \cite{CuFi05}.
Let $\supp{\mu} = \{ v_1, \ldots, v_r \}$ for $v_1, \ldots, v_r \in \re^n$.
We show that $v_i \in \mathcal{S}(H)$ for all $i$.
For each $\xi \in \re^\ell$, it holds that
for all $p \in \re[x]_{k-d_H}$
\[
vec(p) ^T L_{\xi^T H \xi}^{(k)}(y) vec(p) =
\mathscr{L}_y (\xi^T H \xi \cdot p^2) =
\]
\[
\mathscr{L}_y ( (\xi p)^T H (\xi p) ) = vec(p)^T
\left[ (\xi \otimes I)^T \Big( L_{H}^{(k)}(y) \Big) (\xi \otimes I) \right]  vec(p).
\]
(Here $\otimes$ denotes the standard Kronecker product.)
So, we have
\[
L_{\xi^T H \xi}^{(k)}(y) =
(\xi \otimes I)^T \Big( L_H^{(k)}(y) \Big) (\xi \otimes I) \succeq 0.
\]
This implies that $y$ is flat with respect to the inequality
$\xi^T H(x) \xi \geq 0$. Again, by Theorem~1.1 of \cite{CuFi05},
$y$ admits a unique measure and it is
supported in the set $\mathcal{S}(\xi^T H(x) \xi )$.
So, $\xi^T H(v_i) \xi \geq 0$ for all $i$.
This is true for all $\xi \in \re^\ell$,
so $\supp{\mu} \subseteq \mathcal{S}(H)$.
\qed
\end{proof}

\section{The hierarchy of local minimums in $\re^n$}
\setcounter{equation}{0}

This section studies how to find the hierarchy of local minimums
of a polynomial $f$ in the space $\re^n$.
First, we need to compute $H$-minimums.
A point $u \in \re^n$ is called an {\it $H$-minimizer} of $f$
if $\nabla f(u) = 0$ and $\nabla^2 f(u) \succeq 0$.
The set of all $H$-minimizers of $f$
is denoted as $\mc{H}(f)$. If $u \in \mc{H}(f)$,
$f(u)$ is called an $H$-minimum of $f$.
Since each $H$-minimum is a critical value,
the set of all $H$-minimums is also finite. We order them monotonically as
\[
f_1 < f_2 < \cdots < f_N.
\]
The value $f_r$ is called the $r$-th $H$-minimum of $f$.
Since $N$ is typically not known in advance,
we denote $f_{\infty} := \max_{1 \leq k \leq N} f_k$, the biggest $H$-minimum.

\subsection{The smallest $H$-minimum}

Clearly, the smallest $H$-minimum $f_1$
is equal to the optimal value of the problem
\be \label{min:f:Hf>=0}
\min \quad f(x) \quad s.t. \quad
\nabla f(x) = 0, \, \nabla^2 f(x) \succeq 0.
\ee
Since \reff{min:f:Hf>=0} has a polynomial matrix inequality,
we apply the hierarchy of semidefinite relaxations ($k=1,2,\ldots$):
\be \label{<f,y>:gf=0:Hf>=0}
\left\{ \baray{rl}
\vartheta_k^{(1)} : = \min & \langle f, y \rangle \\
 s.t. & \langle 1, y \rangle =1, \, L^{(k)}_{ f_{x_i} } (y)  = 0 \, (i \in [n]),   \\
      & M_k(y) \succeq 0, \, L^{(k)}_{\nabla^2 f} (y) \succeq 0.
\earay \right.
\ee
In the above, $f_{x_i}$ is the partial derivative of $f$
with respect to $x_i$.  In \reff{<f,y>:gf=0:Hf>=0},
the dimension of the decision variable $y$ is
$\binom{n+2k}{2k}$.
This kind of relaxations was
introduced in Henrion and Lasserre~\cite{HL06}.
The dual problem of \reff{<f,y>:gf=0:Hf>=0} is
\be \label{sos:gf+Hf}
\eta_k^{(1)} := \, \max \quad \eta \quad
s.t. \quad f - \eta \in \langle \nabla f \rangle_{2k} + Q_k( \nabla^2 f ).
\ee
(See \S 2 for the notation $\langle \nabla f \rangle_{2k}$ and $Q_k(\nabla^2 f)$.)
The properties of the relaxations \reff{<f,y>:gf=0:Hf>=0}
and \reff{sos:gf+Hf} are summarized as follows.

\begin{thm}
\label{thm:f1:R^n}
Let $\mc{H}(f)$ be the set of $H$-minimizers of $f$
and $f_1$ be the smallest $H$-minimum if it exists.

\bit

\item [(i)] If \reff{<f,y>:gf=0:Hf>=0} is infeasible
for some $k$, then $\mc{H}(f) = \emptyset$
and $f$ has no local minimizers.

\item [(ii)] If $V_{\re}(\nabla f)$ is compact
and $\mc{H}(f) = \emptyset$, then, for all $k$ big enough,
\reff{sos:gf+Hf} is unbounded from above and
\reff{<f,y>:gf=0:Hf>=0} is infeasible.

\item [(iii)] If $V_{\re}(\nabla f)$ is compact
and $\mc{H}(f) \ne \emptyset$,
then, for all $k$ big enough, $\vartheta_k^{(1)} = \eta_k^{(1)} = f_1$.

\item [(iv)] If $V_{\re}(\nabla f)$ is finite
and $\mc{H}(f) \ne \emptyset$, then, for all $k$ big enough,
every optimizer $y^*$ of \reff{<f,y>:gf=0:Hf>=0}
has a truncation $y^*|_{2t}$ that is flat
with respect to $\nabla f=0$ and $\nabla^2 f \succeq 0$.

\end{itemize}

\end{thm}

\begin{rmk}  \label{rmk:thm3.1}
1) For generic $f$, the real variety $V_{\re}(\nabla f)$ is finite (cf.~\cite{NR09}).
So, the assumption that $V_{\re}(\nabla f)$ is compact or finite
is almost always satisfied. In the computation,
we do not need to check whether $V_{\re}(\nabla f)$ is compact or not.
The compactness of $V_{\re}(\nabla f)$ is only used in the
proofs of items (ii)-(iii). It is not clear whether or not the compactness
assumption can be removed, to have same conclusions. \,
\\
2) When $y^*|_{2t}$ is flat, it admits a finite measure $\mu$ supported in
$V_{\re}(\nabla f) \cap \{ \nabla^2 f \succeq 0 \}$ by Prop.~\ref{flat:H>=0}.
For such case,  each point in $\supp{\mu}$ is a minimizer of \reff{min:f:Hf>=0}
and $\vartheta_k^{(1)}=\eta_k^{(1)}=f_1$.
So, (iv) provides a criterion (i.e., $y^*$ has a flat truncation $y^*|_{2t}$)
to check the convergence of $\vartheta_k^{(1)},\eta_k^{(1)}$.
We refer to \cite{HenLas05,Nie-ft}.
Among the set of all optimizers of \reff{<f,y>:gf=0:Hf>=0},
if $\rank\,M_k(y^*)$ is maximum and $y^*|_{2t}$ is flat,
then we can get all $H$-minimizers associated to $f_1$.
When \reff{<f,y>:gf=0:Hf>=0} is solved by primal-dual interior point methods,
an optimizer $y^*$ with $\rank\,M_k(y^*)$ maximum is often computed.
We refer to Laurent \cite[\S6.6]{Lau}.
\end{rmk}

\begin{proof}[Proof of Theorem~\ref{thm:f1:R^n}]  \,

(i) This is obvious, because
\reff{<f,y>:gf=0:Hf>=0} is a relaxation of \reff{min:f:Hf>=0}.
%
%

(ii) Since $V_{\re}( \nabla f)$ is compact,
the ideal $\langle \nabla f \rangle$ is archimedean,
because $- \| \nabla f \|^2 \geq 0$ defines a compact set in $\re^n$.
If $\mc{H}(f) = \emptyset$,
then $-\nabla^2 f(x) \npreceq 0$ for all $x \in V_{\re}(\nabla f)$.
By Corollary~3.16 of Klep and Schweighofer \cite{KlSc10}, we have
\[
- 1 \in \langle \nabla f \rangle + Q(\nabla^2 f).
\]
Thus, for all $k$ big enough, \reff{sos:gf+Hf} is unbounded from above,
which then implies the infeasibility of \reff{<f,y>:gf=0:Hf>=0}
by weak duality.

(iii) By Lemma~2 of \cite{NDS} (or Lemma~3.2 of \cite{Nie-jac}),
there exist disjoint complex varieties $U_0, U_1, \ldots, U_t$ such that
\be \label{dcmp:V(gf)}
V_{\cpx}(\nabla f) = U_0 \cup U_1 \cup \cdots \cup U_t,
\ee
where $U_0 \cap \re^n = \emptyset$, $U_i \cap \re^n \ne \emptyset$
and $f \equiv v_i \in \re$ on $U_i$ for $i=1,\ldots,t$.
Order them as
$v_1 > v_2 > \cdots > v_t$. Corresponding to  \reff{dcmp:V(gf)},
the ideal $\langle \nabla f \rangle$
has a primary decomposition (cf.~\cite[Chapter~5]{Stu02})
\be \label{primry:<gf>}
\langle \nabla f \rangle = G_0 \,\cap \, G_1 \,\cap \,\cdots \,\cap \, G_t
\ee
such that $G_i\subseteq \re[x]$ is an ideal and $U_i = V_{\cpx}(G_i)$,
for $i=1,\ldots,t$. Up to shifting $f$ by a constant,
we can assume that $f_1 =0$. Since $\mc{H}(f) \ne \emptyset$,
there exists an idex $\ell>0$ such that $v_\ell = f_1=0$.

To prove the item (iii), it is enough to show that
there exists $N^*$ such that, for all $\eps>0$, we have
\be \label{f+eps:Q(g+H):1st}
f + \eps  = \sig_\eps + \phi_\eps, \, \quad
\sig_\eps \in  Q_{N^*}( \nabla^2 f ), \quad
\phi_\eps \in \langle \nabla f \rangle_{2N^*}.
\ee
In the following, we show how to construct such desired $\sig_\eps$ and $\phi_\eps$.

\smallskip
For $i=0$, $V_{\re}(G_0) \cap \re^n = \emptyset$.
By Real Nullstellensatz (cf.~\cite[Corollary~4.1.8]{BCR}),
there exists $\tau_0 \in \Sig[x]$ such that
$1 + \tau_0  \in  G_0$. From
$f = \frac{1}{4} (f+1)^2 - \frac{1}{4} (f-1)^2$, we get
\begin{align*}
f & \equiv \sig_0 := \frac{1}{4} \left\{
(f+1)^2 + \tau_0 (f-1)^2   \right\}
\quad \mod \quad G_0.
\end{align*}
If $N_0 \geq \deg(f\tau_0)$, then
\[
\sig_0^\eps := \sig_0 + \eps \in \Sig_{2N_0} \subseteq  Q_{N_0}(\nabla^2 f)
\]
for all $\eps>0$. Let $q_0^\eps := f+\eps - \sig_0^\eps$.
Note that $\eps$ is canceled in the subtraction $f+\eps - \sig_0^\eps$,
so $q_0^\eps$ is independent of $\eps$.
Clearly, $q_0^\eps \in \, G_0$.

\smallskip
For $i=1,\ldots, \ell-1$, $v_i>0$ and $v_i^{-1} f - 1 \equiv 0$ on $U_i=V_{\cpx}(G_i)$.
By Hilbert's Strong Nullstellensatz (cf.~\cite{CLO97}),
$( v_i^{-1} f - 1)^{k_i} \in G_i$ for some $k_i \in \N$.
So,
\[
\baray{rl}
s_i := & \sqrt{v_i} \Big(1 \, + \, \big( v_i^{-1} f - 1 \big)\Big)^{1/2} \\
\, \equiv &  \sqrt{v_i}
\sum_{j=0}^{k_i-1} \binom{1/2}{j} \big(v_i^{-1} f - 1 \big)^j\,\,\,\,
\mod\,\,\,\, G_i\, .
\earay
\]
Let $\sig_i^\eps :=  s_i^2 + \eps$ and
$q_i^\eps = f + \eps - \sig_i^\eps \, \in \, G_i$.
In the subtraction $f + \eps - \sig_i^\eps$,
$\eps$ is canceled, so $q_i^\eps$ is independent of $\eps>0$.

\smallskip
When $i= \ell$, $v_\ell = f_1=0$ and $f \equiv 0$ on $U_\ell = V_{\cpx}(G_\ell)$.
By Hilbert's Strong Nullstellensatz,
$f^{k_\ell} \in G_\ell$ for some $k_\ell \in \N$.
Thus, we have
\[
\baray{rl}
s_\ell^\eps := & \sqrt{\eps} \left(1 + \eps^{-1} f \right)^{1/2} \\
 \equiv & \sqrt{\eps}  \sum_{j=0}^{k_\ell-1} \binom{1/2}{j} \eps^{-j} f^j\,\,\,\,
\mod\,\,\,\, G_\ell  .
\earay
\]
Let $\sig_\ell^\eps := (s_\ell^\eps)^2$ and
$q_\ell^\eps = f + \eps - \sig_\ell^\eps  \, \in \,  G_\ell$.
Note that $f^{k_\ell+j}  \in G_\ell$ for all $j \in \N$, and
\[
q_\ell^\eps = c_0(\eps) f^{k_\ell} + \cdots + c_{2k_\ell-2} (\eps) f^{2k_\ell-2}
\]
for some real scalars $c_j(\eps)$.
The degree of $q_\ell^\eps$ is independent of $\eps>0$.

For $i = \ell +1, \ldots, t$, we have $v_i < f_1 = 0$.
Thus, for all $x\in V_{\re}(G_i)$, we have $x \not\in \mc{H}(f)$, i.e.,
$-\nabla^2 f(x) \npreceq 0$. The ideal
$G_i \supseteq \langle \nabla f\rangle$ is archimedean,
because $V_{\re}(\nabla f)$ is compact.
By Corollary~3.16 of Klep and Schweighofer \cite{KlSc10},
there exists $\tau_i \in Q(\nabla^2 f)$ such that
$ 1 +\tau_i \in  G_i$. From
$f = \frac{1}{4} (f+1)^2 - \frac{1}{4} (f-1)^2$, we get
\begin{align*}
f & \equiv \frac{1}{4} \left\{
(f+1)^2 + \tau_i (f-1)^2   \right\}
\quad \mod \quad G_i.
\end{align*}
Let
$
\sig_i^\eps := \eps + \frac{1}{4} \left\{
(f+1)^2 + \tau_i (f-1)^2   \right\}.
$
Clearly, if $N_1>0$ is big enough, then
\[
\sig_i^\eps \in  Q_{N_1}(\nabla^2 f)
\]
for all $\eps>0$. Let
$q_i^\eps := f+\eps - \sig_i^\eps \, \in \, G_i$.
Since $\eps$ is canceled in the subtraction $f+\eps - \sig_i^\eps$,
$q_i^\eps$ is independent of $\eps > 0$.

\smallskip
The complex varieties of the ideals $G_i$ are disjoint from each other.
Applying Lemma~3.3 of \cite{Nie-jac} to $G_0,G_1,\ldots,G_t$,
we can get $a_0, \ldots, a_r \in \re[x]$ satisfying
{\small
\[
a_0^2+\cdots+a_t^2-1 \, \in \, \langle \nabla f \rangle, \quad a_i \in
\bigcap_{ j \ne i } G_j.
\]
} \noindent
Let $\sig_\eps = \sig_0^\eps a_0^2+\sig_1^\eps a_1^2 +
\cdots + \sig_t^\eps a_t^2$, then
\be
f + \eps  - \sig_\eps  = \sum_{i=0}^t (f+\eps-\sig_i^\eps) a_i^2
+ (f+\eps)(1-a_0^2-\cdots-a_t^2).
\ee
Since $q_i^\eps = f+\eps-\sig_i^\eps \in G_i$, it holds that
{\small
\[
(f+\eps-\sig_i^\eps) a_i^2 \, \in \, \bigcap_{j=0}^t G_j = \langle \nabla f \rangle.
\] \noindent}For all $i \ne \ell$,
we have seen that $q_i^\eps$ is independent of $\eps > 0$.
There exists $N_2 >0$ such that for all $\eps >0$
\[
(f+\eps-\sig_i^\eps) a_i^2 \, \in \, \langle \nabla f \rangle_{2N_2}
\, \mbox{ for all } \, i \ne \ell.
\]
For $i=\ell$, the degree of $q_\ell^\eps=f+\eps-\sig_\ell^\eps$
is independent of $\eps > 0$.
So, there exists $N_3 >0$ such that for all $\eps >0$
\[
(f+\eps-\sig_\ell^\eps) a_\ell^2 \, \in \, \langle \nabla f \rangle_{2N_3}.
\]
Since $1-a_1^2-\cdots-a_t^2 \in \langle \nabla f \rangle$,
there exists $N_4 >0$ such that for all $\eps >0$
\[
(f+\eps)(1-a_1^2-\cdots-a_t^2) \, \in \, \langle \nabla f \rangle_{2N_4}.
\]
Combining the above, we know that if $N^* \geq \max_{0\leq i \leq 4} N_i$, then
\[
\phi_\eps := f + \eps  - \sig_\eps \, \in \, \langle \nabla f \rangle_{2N^*}
\]
for all $\eps >0$.
From the constructions of $\sig_i^\eps$ and $a_i$,
we know their degrees are independent of $\eps$.
So, $\sig_\eps  \, \in \, Q_{N^*}(\nabla^2 f)$
for all $\eps >0$, if $N^*$ is big enough.
Therefore, \reff{f+eps:Q(g+H):1st} is proved.

(iv) Since $V_{\re}(\nabla f)$ is finite, by Proposition~4.6
of Lasserre, Laurent and Rotalski~\cite{LLR08},
there exists $t>0$ such that for
every $y$ that is feasible in \reff{<f,y>:gf=0:Hf>=0},
the truncation $y|_{2t}$ is flat with respect to $\nabla f= 0$.
Since $\deg(\nabla f) = \deg(\nabla^2 f)+1$ and $L_{\nabla^2 f}^{(t)}(y) \succeq 0$,
$y|_{2t}$ is also flat with respect to $\nabla^2 f \succeq 0$.
The conclusion follows since $y^*$ is feasible.
\qed
\end{proof}

\subsection{Bigger $H$-minimums}
\label{subsec:fr+1}

Suppose the $r$-th $H$-minimum $f_r$ is known.
We want to check whether $f_{r+1}$ exists or not;
if it does, we compute it.
For $\dt >0$, consider the problem
\be \label{v+:fk+dt}
\left\{\baray{rl}
\mathscr{H}^{>}(f_r + \dt) := \min & f(x) \\
s.t. & \nabla f(x) = 0, \, \nabla^2 f(x) \succeq 0, \, f(x) \geq f_r+\dt.
\earay \right.
\ee
Clearly, $\mathscr{H}^{>}(f_r + \dt)$ is the smallest
$H$-minimum $\geq f_r+\dt$. If $0< \dt \leq f_{r+1} - f_r$,
then $f_{r+1} = \mathscr{H}^{>}(f_r + \dt)$.
To solve \reff{v+:fk+dt}, we propose the following hierarchy
of semidefinite relaxations ($k=1,2,\cdots$):
\be \label{<f,y>:kth:mom}
\left\{ \baray{rl}
\vartheta_k^{(r+1)} : = \min & \langle f, y \rangle \\
 s.t. & L^{(k)}_{ f_{x_i} } (y)  = 0 \, (i\in [n]), \, L^{(k)}_{f-f_r-\dt} (y) \succeq 0,  \\
 & \langle 1, y \rangle = 1, \, M_k(y) \succeq 0, \, L^{(k)}_{\nabla^2 f} (y) \succeq 0.
\earay \right.
\ee
In \reff{<f,y>:kth:mom}, the dimension of the decision variable
$y$ is $\binom{n+2k}{2k}$.
Its dual problem is
\be \label{sos:kth:gf+Q}
\eta_k^{(r+1)} := \, \max \quad \gamma \quad
s.t. \quad f - \gamma \in \langle \nabla f \rangle_{2k} +
 Q_k(f-f_r-\dt) +  Q_k( \nabla^2 f ).
\ee

\begin{thm}  \label{thm:Hoc:fr}
Let $f_r$ be the $r$-th $H$-minimum of $f$,
and $f_\infty$ be the biggest one.

\bit

\item [(i)] If \reff{<f,y>:kth:mom} is infeasible
for some $k$, then \reff{v+:fk+dt} is infeasible
and $f_r + \dt > f_{\infty}$.

\item [(ii)] If $V_{\re}(\nabla f) \cap \{f(x) \geq f_r + \dt \}$ is compact
and $f_r + \dt > f_{\infty}$, then, for all $k$ big enough,
\reff{sos:kth:gf+Q} is unbounded from above and
\reff{<f,y>:kth:mom} is infeasible.

\item [(iii)] If $V_{\re}(\nabla f) \cap \{f(x) \geq f_r + \dt \}$ is compact
and $f_r + \dt \leq f_{\infty}$,
then, for all $k$ big enough,
$\vartheta_k^{(r+1)} = \eta_k^{(r+1)} = \mathscr{H}^{>}(f_r + \dt)$.

\item [(iv)] If $V_{\re}(\nabla f) \cap \{f(x) \geq f_r + \dt \}$ is finite
and $f_r + \dt \leq f_{\infty}$, then for all $k$ big enough,
every optimizer $y^*$ of \reff{<f,y>:kth:mom}
has a truncation $y^*|_{2t}$ that is flat with respect to $\nabla f = 0$,
$f-f_r -\dt \geq 0$ and $\nabla^2 f \succeq 0$.

\eit

\end{thm}

\begin{rmk}  \label{rmk:thm3.3}
The assumptions in Theorem~\ref{thm:Hoc:fr} are almost always satisfied;
we often get all $H$-minimizers on which $f$ equals $\mathscr{H}^{>}(f_r+\dt)$,
when \reff{<f,y>:kth:mom} is solved by primal-dual interior point methods.
We refer to Remark~\ref{rmk:thm3.1}.
\end{rmk}

\begin{proof}[Proof of Theorem~\ref{thm:Hoc:fr}]\,

(i) This is clear, because \reff{<f,y>:kth:mom} is a relaxation of \reff{v+:fk+dt}.

(ii) By the compactness of $V_{\re}(\nabla f) \cap \{ f(x) \geq f_r + \dt \}$,
we know $\langle \nabla f \rangle + Q( f - f_r -\dt)$ is archimedean.
Since $f_r + \dt > f_{\infty}$, \reff{v+:fk+dt} is infeasible.
So, $-\nabla^2 f(x) \npreceq 0$ for all
$x \in V_{\re}(\nabla f) \cap \{ f(x) \geq f_r + \dt \}$.
By Corollary~3.16 of Klep and Schweighofer \cite{KlSc10}, we get
\[
-1  \in  \langle \nabla f \rangle + Q( f - f_r -\dt) + Q(\nabla^2 f).
\]
So, for $k>0$ big enough, \reff{sos:kth:gf+Q} is unbounded from above
and \reff{<f,y>:kth:mom} is infeasible.

(iii) This can be proved in the same way as for
Theorem~\ref{thm:f1:R^n}(iii).
Here, we only list the differences.
First, we can get the decompositions
\reff{dcmp:V(gf)} and \reff{primry:<gf>}.
Note that $U_0 \cap \re^n = \emptyset$,
$U_i \cap \re^n \ne \emptyset$
and $f \equiv v_i \in \re$ for $i=1,\ldots, t$.
Up to shifting $f$ by a constant,
we can assume $\mathscr{H}^{>}(f_r+\dt) = 0$.
Choose the index $\ell>0$ such that $v_\ell = 0$.
Like \reff{f+eps:Q(g+H):1st}, it is enough to show that
there exists $N^*$ such that, for all $\eps>0$,
\be \label{f+eps:Q(g+H):rth}
\left\{ \baray{c}
f + \eps  = \phi_\eps  + \sig_\eps\, \quad
\phi_\eps \in \langle \nabla f \rangle_{2N^*},   \\
\sig_\eps \in Q_{N^*}(f-f_r-\dt) + Q_{N^*}( \nabla^2 f ).
\earay \right.
\ee
For $i=0, 1, \ldots, \ell$,
we can construct $\sig_i^\eps$ in the same way as in
in the proof of Theorem~\ref{thm:f1:R^n}(iii).

For $i=\ell+1,\ldots,t$, we have $v_i < 0$.
By the assumption, $V_{\re}(G_i) \cap \{ f(x) \geq f_r+\dt\}$ is compact,
so $G_i + Q(f-f_r-\dt)$ is archimedean.
For all $x \in V_{\re}(G_i) \cap \{ f(x) \geq f_r+\dt\}$,
we have $-\nabla^2 f(x) \npreceq 0$,
because otherwise we can get the contradiction
$\mathscr{H}^{>}(f_r+ \dt) \leq v_i <0$.
By Corollary~3.16 of Klep and Schweighofer \cite{KlSc10},
there exists $\tau_i \in Q(f-f_r-\dt)+ Q(\nabla^2 f)$ such that
$1 + \tau_i \in G_i.$
Then, we construct $\sig_i^\eps$ similarly
as in the proof of Theorem~\ref{thm:f1:R^n}(iii).
The rest of the proof is same.

(iv) The proof is same as for Theorem~\ref{thm:f1:R^n}(iv),
by using Remark~4.9 of Lasserre, Laurent and Rotalski~\cite{LLR08},
\qed
\end{proof}

Note that $\mathscr{H}^{>}(f_r+\dt)$ is the smallest
$H$-minimum $\geq f_r+\dt$. So, if $ 0 < \dt < f_{r+1} - f_r$,
then $f_{r+1} = \mathscr{H}^{>}(f_r+\dt)$.
We typically do not know if $\dt < f_{r+1} - f_r$ or not.
Here, we introduce a trick to verify this.
Consider the maximization problem
\be \label{maxf:vk(dt)}
\left\{\baray{rl}
\mathscr{H}^{<}(f_r + \dt) := \max & f(x) \\
s.t. & \nabla f(x) = 0, \, \nabla^2 f(x) \succeq 0, \, f(x) \leq f_r+\dt.
\earay \right.
\ee
The following lemma is obvious.

\begin{lem}
For $\dt>0$, $\mathscr{H}^{<}(f_r + \dt) = f_r$
if and only if $\dt < f_{r+1} - f_r$.
\end{lem}

The optimal value $\mathscr{H}^{<}(f_r + \dt)$
can be computed by solving a hierarchy of semidefinite relaxations
that are similar to \reff{<f,y>:kth:mom}. Similar properties like in
Theorems~\ref{thm:f1:R^n} and \ref{thm:Hoc:fr}
can be proved. For cleanness of the paper, we omit them here.
Once $f_r$ is known, we can determine $f_{r+1}$
by the following procedure:

\bit

\item [0.] Choose a small positive value of $\dt$ (e.g., 0.01).

\item [1.] Compute the optimal value
$\mathscr{H}^{<}(f_r+\dt)$ of \reff{maxf:vk(dt)}.

\item [2.] Solve the hierarchy of semidefinite relaxations \reff{<f,y>:kth:mom}.

\bit

\item [a)]
If \reff{<f,y>:kth:mom} is infeasible for some $k$
and $\mathscr{H}^{<}(f_r+\dt) = f_r$,
then $f_r = f_{\infty}$ and stop.

\item [b)]
If \reff{<f,y>:kth:mom} is infeasible for some $k$
but $\mathscr{H}^{<}(f_r+\dt) > f_r$,
then decrease the value of $\dt$
(e.g., let $\dt:=\dt/2$) and go to Step~1.

\item [c)]
If \reff{<f,y>:kth:mom} is feasible for all $k$,
we generally get $\mathscr{H}^{>}(f_r+\dt)$ for $k$ big.
If $\mathscr{H}^{<}(f_r+\dt) = f_r$,
then $f_{r+1} = \mathscr{H}^{>}(f_r+\dt)$ and stop;
otherwise, decrease the value of $\dt$
(e.g., let $\dt:=\dt/2$) and
go to Step~1.

\eit

\eit

Once $f_{r+1}$ is computed, we use the same procedure to detect
whether $f_{r+2}$ exist or not; if it does, we can get it.
This process can be repeated to get all $H$-minimums.

In the above procedure, a value for $\dt$ satisfying
$0 < \dt < f_{r+1}-f_r$ is found by a kind of bisection process.
Since $f_{r+1}-f_r$ is positive,
this bisection process always terminates in finitely many steps.
In computation, the value of $\dt$ cannot be too small,
because otherwise there are numerical troubles for solving
the resulting semidefinite relaxations. However,
in practice, an initial value like $0.01$ is often small enough.
This is demonstrated by our examples. In applications
where $f_{r+1}-f_r$ is really tiny, a trick of scaling can be applied.
If we multiply $f$ with a constant $C$,
then the gap $f_{r+1}-f_r$ is changed to $C(f_{r+1}-f_r)$.
So, we can choose a big $C$,
then apply the same procedure for $Cf$ with a relatively large $\dt$.
Of course, this might cause numerical troubles sometimes.
We would also like to remark that the gap $f_{r+1}-f_r$
can be arbitrarily small, while the coefficients of
$f$ are not. For instance, consider the univariate polynomial
$
f = x^2(x^2-2x+1+\eps),
$
with a parameter $\eps>0$.
It has two local minimizers, which are respectively
$
0,  (3 + \sqrt{1-8 \eps})/4.
$
The gap is
\[
f_2 - f_1 = \frac{
1- \sqrt{1-8 \eps}^3+20 \eps - 8 \eps^2}{32}.
\]
Clearly, it goes to zero as $\eps \to 0$,
but the $\infty$-norm of the coefficient vector of $f$ is constantly $2$.
So, there does not exist
a uniform lower bound for the gap.

\subsection{Extracting local minimum values}

Once all $H$-minimums $f_1,\ldots, f_N$ are computed,
we can get the hierarchy of local minimums from them.
As mentioned in Remarks~\ref{rmk:thm3.1} and \ref{rmk:thm3.3},
for each $f_r$,
we often get all the associated $H$-minimizers,
when \reff{<f,y>:kth:mom} is solved by primal-dual interior-point methods.

Let $u$ be an $H$-minimizer of $f$, i.e.,
$\nabla f(u) =0$ and $\nabla^2 f(u) \succeq 0$, such that $f(u) = f_r$.
If $\nabla^2 f(u) \succ 0$,
then $f_r$ is a local minimum.
If $\nabla^2 f(u)$ is singular, we cannot make such a judgement.
For $\rho>0$, consider the optimization problem
\be \label{minf:|x-u|<rho}
f_{u,\rho} := \min \quad f(x) \quad
s.t. \quad  \rho^2 - \|x-u\|^2 \geq 0.
\ee
Clearly, $u$ is a local minimizer of $f$
if and only if $f(u) = f_{u,\rho}$ for some $\rho>0$.
This fact can be applied to verify local optimality of $u$.
The optimal value $f_{u,\rho}$ can be computed by
using the Jacobian SDP relaxation method in \cite{Nie-jac}.

\subsection{Examples}

The semidefinite relaxations \reff{<f,y>:gf=0:Hf>=0} and \reff{<f,y>:kth:mom}
can be solved by software {\tt YALMIP}~\cite{YALMIP} on moment relaxations,
which uses the SDP solver {\tt SeDuMi}~\cite{sedumi}.
We first apply the procedure at the end of \S \ref{subsec:fr+1}
to get all $H$-minimums, and then extract the hierarchy of local minimums.

Recall that $f_1 < \cdots < f_N$ are the hierarchy of
$H$-minimum values of the polynomial $f$, defined at the beginning of this section.
The optimal value $\mathscr{H}^{>}(f_r+\dt)$ is
defined in \reff{v+:fk+dt}, and  $\mathscr{H}^{<}(f_r+\dt)$ is
defined in \reff{maxf:vk(dt)}.
Such notation will be used in the following examples.

\begin{exm}
(i) Consider the polynomial $x_1^2 +(x_1x_2-1)^2$.
Its infimum over $\re^2$ is zero,
but it is not achievable (cf. \cite{NDS}).
It does not have a local minimizer, because
the relaxation~\reff{<f,y>:gf=0:Hf>=0} is infeasible for $k=3$.

\noindent
(ii) Consider the polynomial
\[
2x_2^4(x_1+x_2)^4+x_2^2(x_1+x_2)^2 + 2x_2(x_1+x_2)  + x_2^2
\]
Its infimum over $\re^2$ is $-5/8$, which is also not achievable
(cf.~\cite[Example~4.1]{VuiSon}).
This polynomial does not have a local minimizer either,
because the relaxation~\reff{<f,y>:gf=0:Hf>=0} is infeasible for $k=5$.
\end{exm}

\begin{exm}
(i)
Consider the dehomogenized Motzkin polynomial (\cite{Rez00})
\[
1+x_1^4x_2^2 + x_1^2x_2^4 - 3 x_1^2x_2^2.
\]
The first $H$-minimum $f_1 = 0$, achieved at $(\pm 1, \pm 1)$.
The Hessian $\nabla^2 f$ is positive definite on $(\pm 1, \pm 1)$,
so $f_1$ is the smallest local minimum.
The second $H$-minimum $f_2 = 1$, achieved on
the two lines $(t,0)$,$(0,t)$.
Since $(2,0)$ is a local minimizer
(verified by solving \reff{minf:|x-u|<rho}), $f_2$ is the second
local minimum. We know $f_2$ is the biggest $H$-minimum,
because \reff{<f,y>:kth:mom} is infeasible for $(k,\dt)=(3,0.001)$
and $\mathscr{H}^{<}(f_2+0.001)=1$.
The hierarchy of local minimums is $\{0,1\}$. \\
\noindent
(ii)
Consider the dehomogenized Robinson polynomial (\cite{Rez00})
\[
x_1^6+x_2^6 + 1 + 3x_1^2x_2^2 - x_1^4(x_2^2+1) - x_2^4(1+x_1^2) - (x_1^2+x_2^2).
\]
The first $H$-minimum $f_1 = 0$, achieved at points $(\pm 1, 0)$, $(0,\pm 1)$,
$(\pm 1, \pm 1)$. Since $\nabla^2 f(1,1)$ is positive definite,
$f_1$ is the smallest local minimum.
%
%
Because \reff{<f,y>:kth:mom} is infeasible for $(k,\dt)=(6, 0.01)$
and $\mathscr{H}^{<}(f_1+0.01)=0$, we know $f_1$
is the biggest $H$-minimum.
There is only one local minimum value.
\end{exm}

\begin{exm}
Consider the polynomial $f$ given as
\[
\baray{c}
x_1^6+x_2^6+x_3^6+x_4^6 - 5( x_1^3x_2^2+x_2^2x_3^3+x_3x_4^4)
+6(x_1^2x_2^2 + x_3^3x_4 + x_1x_2x_3x_4) \\
- 7(x_1x_2x_3 + x_2x_3x_4)
+ (x_1+x_2+x_3+x_4-1)^2-1.
\earay
\]
The 1st through 5th $H$-minimums $f_1,\ldots,f_5$ are as follows:
\bcen
\btab{|r|r|c|l|}  \hline
$r$ & $f_r$ & $H$-minimizers  & local optimality \\ \hline
 1  & $-1813.2169$  & $(3.0149,  3.3618,  3.7667,  -3.7482)$ & minimizer \\ \hline
 2  & $-1515.4286$  & $(-1.1245, -3.0510, 3.6415, -3.6848)$ &  minimizer\\ \hline
 3  & $-140.8532$ &  $(-0.6017, 2.2670, 2.4317, 2.7935)$ & minimizer \\ \hline
 4  & $-62.7880$  &  $( 2.2031, -2.3876, 2.4169, 2.7577)$ & minimizer \\ \hline
 5  & $-4.3786$  & $(0.8653, -0.3392, -1.2499,  0.7930 )$ & minimizer \\ \hline
\etab
\ecen
They are all local minimums, because
$\nabla^2 f \succ 0$ at the $H$-minimizers.
We know $f_5$ is the biggest $H$-minimum,
because \reff{<f,y>:kth:mom} is infeasible for $(k,\dt)=(4,0.01)$
and $\mathscr{H}^{<}(f_5+0.01)=f_5$.
The hierarchy of local minimums is $\{f_1, f_2, \ldots, f_5\}$.
\end{exm}

\begin{exm} \label{exmp3.9}
Consider the polynomial $f$ given as (cf.~\cite{NR09}):
\[
21 x_2^2-92 x_1  x_3^2-70 x_2^2 x_3-95 x_1^4-47 x_1  x_3^3+
51 x_2^2 x_3^2+47 x_1^5+5 x_1  x_2^4+33 x_3^5.
\]
It is unbounded from below, so it has no global minimizers.
The first and second $H$-minimums $f_1, f_2$ are as follows:
\bcen
\btab{|r|r|c|l|} \hline
$r$ & $f_r$ & $H$-minimizers  & local optimality  \\ \hline
 1  & $-549.9848$ & $( 1.9175, 0.0000,  1.7016 )$ & minimizer \\ \hline
 2  &  $0.0000$ & $( 0.0000, 0.0000,  0.0000)$ &  saddle point\\ \hline
\etab
\ecen
The value $f_1$ is the smallest local minimum, because
$\nabla^2 f \succ 0$ at the $H$-minimizer.
The value $f_2$ is not a local minimum,
because the origin is not a local minimizer.
(When restricted to the line $x_1=x_2=0$, $0$ is not a local minimizer).
The relaxation \reff{<f,y>:kth:mom} is infeasible for $(k,\dt)=(3, 0.1)$,
and $\mathscr{H}^{<}(f_2+0.01)=f_2$. So, $f_2$
is the biggest $H$-minimum.
There is only one local minimum value.
\end{exm}

\section{The hierarchy of local minimums with equality constraints}
\setcounter{equation}{0}

We study how to compute the hierarchy
of local minimums when there are equality constraints.
Consider the problem
\be  \label{minf:h=0}
\min  \quad f(x) \quad s.t. \quad h(x) = 0,
\ee
with $f \in \re[x]$ and $h =(h_1,\ldots, h_m) \in \re[x]^m$.
A point $u$ is a critical point of \reff{minf:h=0} if
there exists $\lmd =(\lmd_1,\ldots, \lmd_m)$,
the vector of Lagrange multipliers, satisfying
\be \label{kkt:f-h:x}
\nabla f(u) \, = \lmd_1 \nabla h_1(u) + \cdots + \lmd_m \nabla h_m(u), \quad
h(u) = 0.
\ee
We call such $(u,\lmd)$ a critical pair and $f(u)$ a critical value.
Clearly, $(u,\lmd)$ is a critical pair if and only if
$(u,\lmd)$ is a critical point of the Lagrangian function
\[
L(x,\lmd) := f(x) - \lmd^T h(x).
\]
So, the problem \reff{minf:h=0} always has finitely many critical values.

Suppose the real variety $V_{\re}(h)$ is smooth, i.e., the gradients
$
\nabla h_1(x), \ldots, \nabla h_m(x)
$
are linearly independent for all $x\in V_{\re}(h)$.
If $u$ is a local minimizer of \reff{minf:h=0},
then there exists $\lmd \in \re^m$
such that $(u,\lmd)$ is a critical pair,
and the second order necessary condition holds:
{\small
\be \label{opcd:sonc}
v^T\left( \nabla^2_x L(u,\lmd) \right) v \geq 0
\qquad \forall \, v \in \bigcap_{i=1}^m \nabla h_i(u)^\perp.
\ee  \noindent}(Denote by 
$a^\perp$ the orthogonal complement to $a$.)
Conversely, if $(u,\lmd)$ is a critical pair and
the second order sufficiency condition holds:
{\small
\be \label{f:h:sosc}
v^T\left( \nabla^2_x L(u,\lmd) \right) v > 0
\qquad \forall \, 0 \ne v \in \bigcap_{i=1}^m \nabla h_i(u)^\perp,
\ee
} \noindent
then $u$ is a strict local minimizer.
In short, for $u$ to be a local minimizer,
\reff{kkt:f-h:x} and \reff{opcd:sonc}
are necessary conditions,
while \reff{kkt:f-h:x} and \reff{f:h:sosc}
are sufficient conditions (cf.~\cite{Brks}).
However, for generic polynomials,
\reff{kkt:f-h:x} and \reff{f:h:sosc}
are sufficient and necessary for $u$ to be a local minimizer
(cf.~\cite{Nie-opcd}).

In this paper, we only consider real critical points
and real critical values.
For convenience, we just call them critical points
and critical values.
We order the critical values of \reff{minf:h=0} monotonically as
\[
c_1 < c_2 < \cdots  < c_N.
\]
The value $c_r$ is called the $r$-th critical value.
Let $c_{\infty} := \max_{1\leq i \leq N} c_i$.
Denote by $\mc{C}(f,h)$ the set of all critical points of \reff{minf:h=0}.
We first compute critical values $c_r$, and then
extract the hierarchy of local minimums from them.

\subsection{The smallest critical value}

Clearly, every critical point $u$ belongs to the determinantal variety
\[
D(f,h) := \left\{ x \in \re^n \mid \rank \bbm \nabla f(x) & \nabla h_1 (x) &
\cdots & \nabla h_m(x) \ebm \leq m \right\}.
\]
When $V_{\re}(h)$ is smooth, every point in $D(f,h) \cap V_{\re}(h)$
is a critical point.
We consider the general case that $m<n$ and $V_{\re}(h)$ is smooth.
(When $m=n$, the feasible set in \reff{minf:h=0} is generically a finite set,
and each feasible point is critical.)
Let $\phi_1,\ldots,\phi_K$ be a minimum set of defining polynomials for $D(f,h)$.
As shown in \cite[Section~2]{Nie-jac},
\[
K =(m+1)(n-m-1)+1
\]
and the polynomials $\phi_j$ can be chosen as
{\small
\be \label{def:phi}
\phi_j = \sum_{ \substack{1 \leq i_1 < \cdots < i_{m+1} \leq n \\
i_1 + \cdots + i_{m+1} = \half (m+1)(m+2)+j-1 }  }
\det ( J_{i_1,\ldots,i_{m+1} } )
\ee \noindent}for
$j=1,\ldots, K$, where $J_{i_1,\ldots,i_{m+1} }$ denotes the
$(m+1)$-by-$(m+1)$ submatrix of
$\bbm \nabla f(x) & \nabla h_1 (x) & \cdots & \nabla h_m(x) \ebm$
with row indices $i_1,\ldots,i_{m+1}$. For convenience, let
\[
\phi := (\phi_1, \ldots, \phi_K).
\]
When $V_{\re}(h)$ is smooth, each critical value $c_i$ of \reff{minf:h=0}
is the objective value of a feasible point of the optimization problem
\be  \label{minf:h=0:det}
\min  \quad f(x) \quad
s.t. \quad  h(x)  = 0, \, \,  \phi(x) = 0,
\ee
and vice versa (cf.~\cite{Nie-jac}).

Apply the hierarchy of semidefinite relaxations to solve \reff{minf:h=0:det}
($k=1,2,\cdots$):
\be  \label{las-r:h=0:det}
\left\{ \baray{rl}
\zeta_k^{(1)} := \min  &  \langle f, y \rangle \\
\mbox{s.t.} & L_{h_i}^{(k)}(y)= 0\, (i\in [m]), \,
   L_{\phi_j}^{(k)}(y) =0\,(j \in [K]), \\
&  \langle 1, y \rangle = 1, \,  M_k(y) \succeq 0.
\earay \right.
\ee
The dual problem of \reff{las-r:h=0:det} is
\be \label{sos-las:h=0:det}
\theta_k^{(1)} := \max  \quad \gamma   \quad
\mbox{s.t.} \quad f - \gamma \in \Sig[x]_{2k} +
\langle h \rangle_{2k} + \langle \phi \rangle_{2k}.
\ee

\begin{thm}  \label{thm:min:c1}
Suppose $V_{\re}(h) \ne \emptyset$ is smooth.
Let $\mc{C}(f,h)$ be the set of critical points of \reff{minf:h=0},
and $c_1$ be the smallest critical value if it exists.

\bit

\item [(i)] The set $\mc{C}(f,h) = \emptyset$
if and only if \reff{las-r:h=0:det} is infeasible for some $k$.

\item [(ii)] If $\mc{C}(f,h) \ne \emptyset$,
then, for all $k$ big enough,
$\theta_k^{(1)} = \zeta_k^{(1)} = c_1$.

\item [(iii)] If $\mc{C}(f, h) \cap \{f(x) = c_1\} \ne \emptyset$ is finite,
then, for all $k$ big enough,
every optimizer $y^*$ of \reff{las-r:h=0:det}
has a truncation $y^*|_{2t}$ that is flat
with respect to $h=0$ and $\phi = 0$.

\eit
\end{thm}

\begin{rmk} \label{rmk:thm4.1}
For generic $(f,h)$, the set $\mc{C}(f,h)$ is finite (cf.~\cite{NR09}).
A good criterion of checking the convergence of $\theta_k^{(1)}, \zeta_k^{(1)}$
is that $y^*$ has a flat truncation $y^*|_{2t}$.
When \reff{las-r:h=0:det} is solved by primal-dual interior point methods,
we often get all the minimizers of \reff{minf:h=0:det},
which are critical points associated to $c_1$.
We refer to Remarks~\ref{rmk:thm3.1} and \ref{rmk:thm3.3}.
\end{rmk}

\begin{proof}[Proof of Theorem~\ref{thm:min:c1}] \,

(i) Note that \reff{las-r:h=0:det} is a relaxation of \reff{minf:h=0:det}.
If \reff{las-r:h=0:det} is infeasible for some $k$,
then \reff{minf:h=0:det} must be infeasible,
which implies that $\mc{C}(f,h) = \emptyset$.

Conversely, if $\mc{C}(f,h) = \emptyset$,
then \reff{minf:h=0:det} is infeasible, because $V_{\re}(h)$ is smooth.
By Real Nullstellensatz (cf.~\cite[Corollary~4.1.8]{BCR}), we have
\[
-1 \in  \Sig[x] + \langle h \rangle + \langle \phi \rangle.
\]
This implies that for all $k$ big, \reff{sos-las:h=0:det} is unbounded from above
and hence \reff{las-r:h=0:det} is infeasible.

(ii)-(iii) \,
Since $V_{\re}(h)$ is smooth,
every feasible point of \reff{minf:h=0:det} is a critical point,
and its objective value is a critical value.
The minimum value of \reff{minf:h=0:det} is the smallest
critical value $c_1$ of \reff{minf:h=0}.
Since \reff{minf:h=0} has no inequality constraints,
the relaxations \reff{las-r:h=0:det} and \reff{sos-las:h=0:det}
is equivalent to the Jacobian SDP relaxations
(2.8) and (2.11) in \cite{Nie-jac}.
Therefore, the item (ii) can be implied by Theorem 2.3 of \cite{Nie-jac},
and the item (iii) can be implied by Corollary 4.3 of \cite{Nie-ft}.
\qed
\end{proof}

\subsection{Bigger critical values}
\label{subsec:cr+1}

Suppose the $r$-th critical value $c_r$ of \reff{minf:h=0} is known.
We want to compute the next bigger one $c_{r+1}$,
if it exists. For $\dt >0$, consider the optimization problem
\be  \label{mf:h=0:f>ck}
\left\{ \baray{rl}
\mathscr{C}^{>}(c_r+\dt):= \underset{x\in \re^n}{\min} & f(x) \\
s.t. & h(x)  = 0,\,  \phi(x) = 0, \, f(x) \geq c_r + \dt.
\earay \right.
\ee
Clearly, $\mathscr{C}^{>}(c_r+\dt)$ is the smallest critical value $\geq c_r+\dt$.
We apply the hierarchy of semidefinite relaxations to solve \reff{mf:h=0:f>ck}
($k=1,2,\cdots$):
\be  \label{las-r:h=0:f>ck}
\left\{ \baray{rl}
\zeta_k^{(r+1)} :=\min  &  \langle f, y \rangle \\
s.t. & L_{h_i}^{(k)}(y)= 0\, (i\in [m]), \,
   L_{\phi_j}^{(k)}(y) = 0\, (j \in [K]), \\
& \langle 1, y \rangle = 1, \, \, M_k(y) \succeq 0, \, L_{f-c_r-\dt}^{(k)}(y) \succeq 0.
\earay \right.
\ee
The dual problem of \reff{las-r:h=0:f>ck} is
\be \label{sos-r:h+Q(f>ck)}
\theta_k^{(r+1)} := \max  \quad \gamma   \quad
s.t. \quad f - \gamma \in \langle h \rangle_{2k}
+ \langle \phi \rangle_{2k} + Q_k(f-c_r-\dt).
\ee
The properties of \reff{las-r:h=0:f>ck} and \reff{sos-r:h+Q(f>ck)}
are summarized as follows.

\begin{thm}  \label{thm:rth:cr}
Suppose $V_{\re}(h) \ne \emptyset$ is smooth.
Let $\dt > 0$, $\mc{C}(f,h)$ be the set of critical points of \reff{minf:h=0}
and $c_r$ be the $r$-th smallest critical value of \reff{minf:h=0}
if it exists.

\bit

\item [(i)] The problem \reff{mf:h=0:f>ck} is infeasible
(i.e., $c_{\infty} < c_r + \dt$) if and only if
\reff{las-r:h=0:f>ck} is infeasible
for some $k$.

\item [(ii)] If $c_r + \dt \leq c_{\infty}$,
then, for all $k$ big enough,
\[
\theta_k^{(r+1)} = \zeta_k^{(r+1)} = \mathscr{C}^{>}(c_r+\dt).
\]

\item [(iii)] If $\mc{C}(f,h) \cap \{ f(x) \geq c_k + \dt \}$ is finite
and $c_r + \dt \leq c_{\infty}$, then for all $k$ big enough,
every optimizer $y^*$ of \reff{<f,y>:kth:mom}
has a truncation $y^*|_{2t}$ that is flat with respect to
$h=0, \phi=0$ and $f - c_r - \dt \geq 0$.

\eit

\end{thm}

\begin{rmk} \label{rmk:thm4.3}
The set $\mc{C}(f,h)$ is finite for generic $(f,h)$.
When \reff{las-r:h=0:f>ck} is solved by primal-dual interior point methods,
we often get all the critical points on which $f \geq c_r + \dt$.
We refer to Remark~\ref{rmk:thm4.1}.
\end{rmk}

\begin{proof}[Proof of Theorem~\ref{thm:rth:cr}]\,

(i) If \reff{las-r:h=0:f>ck} is infeasible for some $k$,
then \reff{mf:h=0:f>ck} must be infeasible,
because \reff{las-r:h=0:f>ck} is a relaxation of \reff{mf:h=0:f>ck}.

Conversely, if \reff{mf:h=0:f>ck} is infeasible,
then $c_r + \dt > c_\infty$ and
the feasible set of \reff{mf:h=0:f>ck} is empty.
By Positivstellensatz (cf.~\cite[Corollary~4.4.3]{BCR}),
\[
-1 \in \langle   h \rangle + \langle \phi \rangle + Q(f-c_r-\dt).
\]
%
%
This implies that \reff{sos-r:h+Q(f>ck)} is unbounded from above
for all big $k>0$, which then implies that
\reff{las-r:h=0:f>ck} is infeasible, by weak duality.

(ii) By weak duality, it holds that for all $k$
\[
\theta^{(r+1)}_k \leq \zeta^{(r+1)}_k \leq
\mathscr{C}^{>}(c_r+\dt).
\]
It is enough to show that
there exists $N^*>0$ such that, for all $\eps>0$,
\be \label{f:cr+1:ufmbd}
f - \left( \mathscr{C}^{>}(c_r+\dt) - \eps \right) \in \langle h \rangle_{2N^*}
+ \langle \phi \rangle_{2N^*}
+ Q_{N^*}( f - c_r - \dt),
\ee
because \reff{f:cr+1:ufmbd} implies that
$\theta^{(r+1)}_k  = \mathscr{C}^{>}(c_r+\dt)$ for all $k\geq N^*$.

\medskip
Let $T=\{ x\in \re^n \mid f(x) \geq c_r + \dt\}$.
and $W = V_{\cpx}(h,\phi)$. By Lemma~3.2 of \cite{Nie-jac},
we can decompose $W$ into the union of disjoint complex varieties
\be \label{dcmp:W}
W = W_0 \cup W_1 \cup \cdots \cup W_t,
\ee
such that $W_0 \cap T = \emptyset$,
$W_i \cap T \ne \emptyset$ and $f \equiv v_i \in \re$ on $W_i$
for $i = 1, \ldots, t$. Order them as
$v_1 > v_2 > \cdots > v_t$.
Then, $\mathscr{C}^{>}(c_r + \dt)= v_t$. Up to shifting $f$ by a constant,
we can further assume that $\mathscr{C}^{>}(c_r + \dt)=v_t=0$.
Corresponding to \reff{dcmp:W}, the ideal
$\langle h \rangle + \langle \phi \rangle$
has a primary decomposition (cf.~\cite[Chapter~5]{Stu02})
\[
\langle h \rangle + \langle \phi \rangle =
E_0 \,\cap \, E_1 \,\cap \,\cdots \,\cap \, E_t
\]
such that each $E_i\subseteq \re[x]$ is an ideal and $W_i = V_{\cpx}(E_i)$.

\smallskip
For $i=0$, $V_{\re}(E_0) \cap T = \emptyset$.
%
%
By Positivstellensatz \cite[Corollary~4.4.3]{BCR},
there exists $\tau \in Q(f-c_r-\dt)$ such that
$1 + \tau  \in E_0$. From
$f = \frac{1}{4} (f+1)^2 - \frac{1}{4} (f-1)^2$, we get
\[
f  \equiv \frac{1}{4} \left\{
(f+1)^2 + \tau (f-1)^2   \right\}
\quad \mod \quad E_0.
\]
Let
$ \sig_0^\eps = \eps + \frac{1}{4} \left\{ (f+1)^2 + \tau (f-1)^2   \right\}$.
For any $N_0 > \deg(f\tau)$, we have
\[ \sig_0^\eps \in Q_{N_0}(f-c_r-\dt) \]
for all $\eps>0$.
Let $q_0^\eps := f+\eps - \sig_0^\eps \, \in \, E_0$.
Since $\eps$ is canceled in the substraction $f+\eps - \sig_0^\eps$,
$q_0^\eps$ is independent of $\eps$.

\smallskip
For $i=1,\ldots, t-1$, we have $v_i>0$
and $v_i^{-1} f - 1 \equiv 0$ on $W_i = V_{\cpx}(E_i)$.
By Hilbert's Strong Nullstenllensatz \cite{CLO97},
$( v_i^{-1} f - 1)^{k_i} \in E_i $ for some $k_i > 0$. Let
\[
\baray{rl}
s_i := & \sqrt{v_i} \left(1 \, + \, \big( v_i^{-1} f   - 1 \big)\right)^{1/2} \\
    \equiv & \,  \sqrt{v_i}  \sum_{j=0}^{k_i-1} \binom{1/2}{j}
\left( v_i^{-1} f  - 1 \right)^j \quad \mod \quad E_i,
\earay
\]
and $\sig_i^\eps :=  s_i^2 + \eps$.
Let $q_i^\eps := f + \eps - \sig_i^\eps$.
The subtraction $f + \eps - \sig_i^\eps$ cancels $\eps$,
so $q_i^\eps$ is independent of $\eps$.

\smallskip
When $i=t$, $f \equiv 0$ on $W_t =V_{\cpx}(E_t)$.
By Hilbert's Strong Nullstenllensatz,
$f^{k_t} \in E_t$ for some $k_t > 0$. So, we get
{\small
\[
s_t^\eps := \sqrt{\eps} \left(1 + f/\eps \right)^{1/2}
 \equiv \sqrt{\eps}  \sum_{j=0}^{k_t-1} \binom{1/2}{j} \eps^{-j} f^j
\quad \mod \quad E_t .
\]
} \noindent
Let $\sig_t^\eps := (s_r^\eps)^2$ and
$q_t^\eps := f + \eps - \sig_t^\eps  \, \in \,  E_t$.
The coefficients of $q_t$ depend on $\eps$, but its degree does not.
This is because
\[
q_t^\eps \, = \,  c_0(\eps) f^{k_t} + \cdots +  c_{k_t-2}(\eps) f^{2k_t-2}
\]
for real numbers  $c_j(\eps)$ and
$f^{k_t+j}  \in E_t$ for all $j \in \N$.

\smallskip
The complex varieties $E_0,E_1,\ldots,E_t$ are disjoint from each other.
By Lemma~3.3 of \cite{Nie-jac},
there exist $a_0, \ldots, a_t \in \re[x]$ such that
{\small
\[
a_0^2+\cdots+a_t^2-1 \, \in \, \langle h \rangle + \langle \phi \rangle,
\quad a_i \in \bigcap_{ j \ne i } E_j.
\]
} \noindent
Let $\sig_\eps = \sig_0^\eps a_0^2+\sig_1^\eps a_1^2 + \cdots + \sig_t^\eps a_t^2$,
then
\[
 f + \eps  - \sig_\eps
 = \sum_{i=0}^t (f+\eps-\sig_i^\eps) a_i^2 + (f+\eps)(1-a_0^2-\cdots-a_r^2).
\]
By repeating the same argument as in end of
the proof of Theorem~\ref{thm:f1:R^n}(iii),
we can show that if $N^*$ is big enough,
then, for all $\eps>0$,
\[
f + \eps  \in \langle h \rangle_{2N^*} +
\langle \phi \rangle_{2N^*} + Q_{N^*}(f-c_r-\dt).
\]
So, \reff{f:cr+1:ufmbd} is proved,
and hence the item (ii) is true.

(iii) This can be implied by Theorem~2.6 of \cite{Nie-ft},
because the hierarchy of \reff{sos-r:h+Q(f>ck)}
has finite convergence and \reff{mf:h=0:f>ck} has finitely many minimizers.
\qed
\end{proof}

Clearly, if $0 < \dt < c_{r+1}- c_r$,
then $\mathscr{C}^{<}(c_r+\dt) = c_{r+1}$.
To check whether $\dt < c_{r+1} - c_r$ or not,
we consider the maximization problem
\be  \label{max:f:h=0:f>ck}
\left\{ \baray{rl}
\mathscr{C}^{<}(c_r+\dt):= \max & f(x) \\
\mbox{s.t.} & h(x)  = 0, \,  \phi(x) = 0, \,  c_r + \dt - f(x) \geq 0.
\earay \right.
\ee
The following lemma is obvious.

\begin{lem}
For $\dt>0$,
$\mathscr{C}^{<}(c_r + \dt) = c_r$
if and only if $\dt < c_{r+1} - c_r$.
\end{lem}

The optimal value $\mathscr{C}^{<}(c_r+\dt)$ can also be computed by
solving a hierarchy of semidefinite relaxations
that are similar to \reff{las-r:h=0:f>ck}.
Similar properties like in Theorem~\ref{thm:rth:cr} hold.
For cleanness, we omit them here.
Once $c_r$ is known, $c_{r+1}$ can be determined
by the following procedure:
\bit

\item [0.] Choose a small positive value of $\dt$ (e.g., $0.01$).

\item [1.] Compute the optimal value $\mathscr{C}^{<}(c_r+\dt)$ of \reff{max:f:h=0:f>ck}.

\item [2.] Solve the hierarchy of semidefinite relaxations \reff{las-r:h=0:f>ck}.

\bit

\item
If \reff{las-r:h=0:f>ck} is infeasible for some $k$
and $\mathscr{C}^{<}(c_r+\dt) = c_r$,
then $c_r = c_{\infty}$ and stop.

\item
If \reff{las-r:h=0:f>ck} is infeasible for some $k$
but $\mathscr{C}^{<}(c_r+\dt) > c_r$,
then decrease the value of $\dt$ (e.g., $\dt :=\dt/2$)
and go to Step 1.

\item
If \reff{las-r:h=0:f>ck} is feasible for all $k$,
then we generally get $\mathscr{C}^{>}(c_r+\dt) = \zeta_k^{(r+1)}$ when $k$ is big.
If  $\mathscr{C}^{<}(c_r+\dt) = c_r$,
then $c_{r+1} = \mathscr{C}^{>}(c_r+\dt)$ and stop;
otherwise, decrease the value of $\dt$ (e.g., $\dt :=\dt/2$) and go to Step 1.

\eit

\eit

After $c_{r+1}$ is obtained, we can use the same procedure
to determine $c_{r+2}$. By repeating this process,
all critical values can be obtained.
Typically, it is hard to estimate the gap $c_{r+1}-c_r$.
There are numerical troubles when $c_{r+1}-c_r$ is very small.
We refer to the discussion at the end of \S3.2.

\subsection{Extracting local minimums}

Suppose all critical values $c_1, \ldots, c_N$ are computed.
We want to check whether they are local minimums or not.
By Remarks~\ref{rmk:thm4.1} and \ref{rmk:thm4.3},
we often get all critical points associated to each $c_r$.

Let $u$ be a critical point on which $f(u) = c_r$.
Clearly, if $u$ satisfies the second order sufficiency condition \reff{f:h:sosc},
then $u$ is a strict local minimizer.
Similarly, if $u$ violates the second order necessary condition \reff{opcd:sonc},
then $u$ is not a local minimizer.
For generic cases, \reff{f:h:sosc} is sufficient
and necessary for $u$ to be a local minimizer (cf.~\cite{Nie-opcd}).
The resting, but also difficult, case is that
$u$ satisfies \reff{opcd:sonc} but not \reff{f:h:sosc}.
Note that $u$ is a local minimizer of \reff{minf:h=0}
if and only if $f(u)$ is the optimal value of
\be \label{mf:h=0:|x-u|<rho}
\min \quad f(x) \quad s.t. \quad
h(x) = 0, \,\rho^2 - \| x - u \|^2 \geq 0
\ee
for a small $\rho>0$.
The Jacobian SDP relaxation method in \cite{Nie-jac}
can be applied to solve \reff{mf:h=0:|x-u|<rho}.
Therefore, the local optimality of $u$
can be verified by solving \reff{mf:h=0:|x-u|<rho}
for a small $\rho >0$.
For cleanness, we omit the details here.

Note that a critical point $u$ is a local maximizer of \reff{minf:h=0}
if and only if $u$ is a local minimizer of $-f$ over $h=0$.
So, the local maximality can also be checked
by similar conditions like \reff{opcd:sonc} and \reff{f:h:sosc}.
If $u$ is neither a local minimizer nor a local maximizer,
then $u$ is a saddle point.

\subsection{Examples}

The semidefinite relaxations \reff{las-r:h=0:det}, \reff{las-r:h=0:f>ck}
can be solved by software
{\tt GloptiPoly~3}~\cite{GloPol3} and {\tt YALMIP}~\cite{YALMIP},
which uses the SDP solver {\tt SeDuMi}~\cite{sedumi}.
We first apply the procedure at the end of \S \ref{subsec:cr+1}
to get all critical values, then decide their local optimality.

Recall that $c_1 < \cdots < c_N$ are the hierarchy of
critical values of the optimization problem \reff{minf:h=0},
defined at the beginning of \S4.
The optimal value $\mathscr{C}^{>}(c_r+\dt)$ is
defined in \reff{mf:h=0:f>ck}, and  $\mathscr{C}^{<}(c_r+\dt)$ is
defined in \reff{max:f:h=0:f>ck}.
Such notation will be used in the following examples.

\begin{exm}
Consider the polynomials
\[
f = x_1^4x_2^2 + x_1^2x_2^4 + x_3^6 - 3 x_1^2 x_2^2 x_3^2, \quad
h = x_1^2+x_2^2+x_3^2 - 1.
\]
The objective is the Motzkin polynomial.
The 1st through $4$th critical values $c_1,\ldots,c_4$
on the unit sphere are computed as follows:
\bcen
\btab{|r|r|c|l|} \hline
$r$ & $c_r$ & critical points &  local optimality \\ \hline
1 & $0.0000$    & $( \pm 0.5774,  \pm 0.5774,   \pm 0.5774)$   &  minimizer \\
  &          & $( \pm 1.0000, 0.0000,  0.0000)$   &  minimizer \\
  &        & $(0.0000,  \pm 1.0000,  0.5774)$   &  minimizer \\ \hline
2 & $0.0156$    & $( \pm 0.2623,  \pm 0.8253,  0.5000)$   &  saddle point  \\
  &       & $( \pm 0.8253,  \pm 0.2623,  0.5000)$   &  saddle point  \\  \hline
3 & $0.2500$   & $( \pm 0.7071,  \pm 0.7071,  0.0000)$   &  maximizer     \\ \hline
4 & $1.0000$   & $(0.0000, 0.0000,  \pm 1.0000)$   &  maximizer  \\ \hline
\etab
\ecen
Because \reff{las-r:h=0:f>ck} is infeasible
for $(k,\dt) = (3,0.01)$
and $\mathscr{C}^{<}(c_4+0.01) = 1$,
we know $c_4$ is the biggest critical value.
There are four critical values,
with one local minimum value
and two local maximum values.
\end{exm}

\begin{exm}
Consider the polynomials given as
\[
\baray{c}
f = (2x_1^4-3x_2^2+4x_3^4)^2-5(x_1x_2-x_2x_3+x_3x_1)^4 +x_1^7x_2+x_2^7x_3+x_3^7x_1, \\
h = x_1^4+x_2^4+x_3^4 -1.
\earay
\]
The 1st through $9$th critical values
$c_1,\ldots,c_9$ are computed as follows:
\bcen
\btab{|r|r|c|l|} \hline
$r$ & $c_r$ & critical points &  local optimality \\ \hline
1 & $-45.0451$ &  $\pm( -0.7599,  0.7624, 0.7572)$  & minimizer \\ \hline
2 & $-1.5552 $ &  $\pm(0.3134, 0.8435,  -0.8341)$  &  minimizer \\  \hline
3 & $-1.1143$  &  $\pm(0.9070, 0.7137, 0.5027)$   &  minimizer \\  \hline
4 & $-0.3788$  &  $\pm(0.8754,  -0.7359,  0.5879)$ &  saddle point \\ \hline
5 & $-0.3650$  &  $\pm(0.8900,  0.7696, -0.3840)$  &  saddle point \\  \hline
6 & $ 0.3554$  &  $\pm(0.4837,  0.8234,  0.8347)$  &  saddle point \\ \hline
7 & $ 4.0191$  &  $\pm(0.9998, 0.0363, 0.1666)$  &  saddle point \\ \hline
8 & $ 9.1456$  &  $\pm(0.0713,  0.9996,  0.1941)$  &  maximizer \\ \hline
9 & $16.1706$  &  $\pm(0.2273, 0.0036, 0.9993)$   &  maximizer \\ \hline
\etab
\ecen
Because \reff{las-r:h=0:f>ck} is infeasible for $(k,\dt) = (5,0.01)$
and $\mathscr{C}^{<}(c_{9}+0.01) = c_{9}$,
we know $c_{9}$ is the biggest critical value.
There are nine critical values,
three of which are local minimum values
and two of which are local maximum values.
\end{exm}

\begin{exm}
Consider the polynomials given as:
\[
\baray{c}
f = x_1^5+x_2^5+x_3^5 + (x_1x_2-x_2x_3-x_3x_1)^2 + (x_1+x_2-x_3)^3, \\
h_1 = 6x_1^4 - 2x_2^4 - 3x_3^4 - 1, \quad
h_2 = 4x_1^2+ 5x_2^2 - 7x_3^2 - 2.
\earay
\]
The 1st through $7$th critical values
$c_1,\ldots,c_7$ are computed as follows:
\bcen
\btab{|r|r|c|l|}  \hline
$r$ & $c_r$ & critical points &  local optimality \\ \hline
1 & $-97.9193$  & $(-2.3943, -2.3119,  2.6092)$ & minimizer \\  \hline
2 &  $-0.6117$  & $( -0.6494,  -0.4021,  -0.2660)$ & maximizer  \\ \hline
3 &  $-0.2008$ & $(-0.6493, 0.4011, 0.2648 )$ &   minimizer   \\ \hline
4 &  $0.1712$ & $( 0.6413, -0.2921,   0.1012)$ &  minimizer \\ \hline
5 &  $0.6121$  & $( 0.6486,  0.3952,  0.2573)$  &  minimizer  \\ \hline
6 &  $1.0710$  & $( 1.1049, -1.1056,  -1.1336)$  &  maximizer \\ \hline
7 &  $1.0843$  &  $(-1.0948, 1.0956, -1.1210 )$ &  maximizer \\ \hline
\etab
\ecen
Since \reff{las-r:h=0:f>ck} is infeasible
for $(k, \dt)= (5, 0.01)$
and $\mathscr{C}^{<}(c_{7}+0.01) = c_{7}$,
we know $ c_7$ is the biggest critical value.
There are seven critical values.
Four of them are local minimum values,
and the resting three are local maximizers.
\end{exm}

\begin{exm}
Consider the polynomials given as
\[
\baray{c}
f = (2x_1^2-x_1x_2-3x_1x_3)y_1^2 +(3x_2^2+2x_2x_1-5x_2x_3)y_1y_2
+(4x_3^2+x_3x_1+3x_3x_2)y_2^2, \\
h_1 =x_1^2+x_2^2+x_3^2-1, \quad h_2 = y_1^2+y_2^2 -1.
\earay
\]
The 1st through $13$th critical values
$c_1,\ldots, c_{13}$ are computed as follows:
\bcen
\btab{|r|r|c|l|} \hline
$r$ & $c_r$ & critical points &  local optimality \\ \hline
1 & $-2.4943$  &  $(0.1503,  0.9095,  -0.3875,  -0.6591,  0.7521 )$  & minimizer   \\  \hline
2 & $-0.8949$  &  $(0.4719,  0.3089,   0.8258,  0.9973,  0.0728 )$  &  minimizer   \\  \hline
3 & $-0.7232$  &  $(0.4199,  -0.2395,   0.8754,   0.9897,  -0.1430 )$ & saddle point \\  \hline
4 & $-0.6003$  &  $(0.4452,   0.7828,   0.4347,   0.9871,  -0.1601)$ & saddle point \\  \hline
5 & $-0.0095$  &  $(0.9254,  -0.3788,   0.0133,   0.06015,   0.9982)$ & saddle point \\  \hline
6 & $0.7898$  &  $(0.9297,  -0.0693,   0.3617,   0.7088,   0.7054)$  & saddle point \\  \hline
7 & $1.1474$  & $(-0.4288,   0.83738,   0.3391,   0.7427,   0.6697)$  & saddle point \\  \hline
8 & $1.3137$  & $(-0.5357,   0.8397,  -0.0893,   0.8780,   0.4786 )$ & saddle point \\  \hline
9 & $1.4812$  & $(0.314,  0.9344,   0.1678,   0.5724,   0.8200)$  & saddle point \\  \hline
10& $2.4943$  & $( 0.6172,   0.4448,  -0.6490,  0.8488,   0.5287 )$  & saddle point \\  \hline
11& $2.8665$  &  $( -0.7019,   0.3469,   0.6221,   0.8785,  -0.4777 )$  & saddle point \\  \hline
12& $2.9211$  & $(0.8469, -0.2327,  -0.4781,   0.9825,  -0.1863 )$   & maximizer \\  \hline
13& $4.6163$  & $( 0.0874,  0.3330,   0.9389,  -0.1193,  0.9929 )$  & maximizer \\  \hline
\etab
\ecen
We know that $c_{13}$ is the biggest critical value,
because \reff{las-r:h=0:f>ck} is infeasible
for $(k,\dt) = (4, 0.01)$
and $\mathscr{C}^{<}(c_{13}+0.01) = c_{13}$.
There are thirteen critical values,
with two local minimum values
and two local maximum values.
\end{exm}

\section{Some extensions and discussions}

\subsection{Critical values in $\re^n$}

The approach in \S4 can be applied to get
all critical values of a polynomial in the space $\re^n$.
When there are no constraints,
the polynomials $\phi_j$ in \reff{def:phi}
are just the partial derivatives of $f$.
Thus, if the $r$-th critical value $c_r$ is computed,
the next bigger one $c_{r+1}$ is the optimal value of
\[
\min \quad f(x) \quad s.t. \quad \nabla f(x) = 0, \quad
f(x)  \geq c_r + \dt,
\]
for some $\dt>0$ sufficiently small.
A similar version of the relaxation \reff{las-r:h=0:f>ck}
can be applied to compute $c_{r+1}$.
The properties in Theorem~\ref{thm:rth:cr} also hold.

\subsection{Local minimums in an open set}

Suppose we want to compute
the hiearchy of local minimums of a polynomial $f$
in an open semialgebraic set of the form
\[
\Omega = \{ x \in \re^n:\, g_1(x) >0, \ldots, g_m(x)>0\}.
\]
The necessary local optimality conditions in an open set
are still $\nabla f(x) = 0$, $\nabla^2 f(x) \succeq 0$.
So, the $H$-minimums are feasible objective values of the problem
\[
\min \quad f(x) \quad s.t. \quad
\nabla f(x) =0, \, \nabla^2 f(x) \succeq 0,  \,
g_i(x) \geq 0 \,(i\in [m]).
\]
The method in Section~3 can be applied to get $H$-minimums.
Once they are obtained, we then check whether they are local minimums.

\subsection{Critical values in a closed set}

We consider to compute the critical values of a polynomial $f$
on a basic closed semialgebraic set.
For convenience, consider the closed set
\[
K = \left\{ x \in \re^n\,
\left| \baray{c}
 g(x) \geq 0
\earay \right.
\right\}
\]
defined by a single polynomial $g$.
A critical point $u$ of $f$ on $K$ satisfies
\[
\nabla f(u) = \mu \nabla g(u),\quad \mu g(u) = 0.
\]
When $V_{\re}(g)$ is nonsingular,
$u$ is a critical point if and only if
\[
 g(u) \cdot \nabla f(u) \  = 0, \quad
\rank\,   \bbm \nabla f(u) &  \nabla g(u) \ebm \leq 1.
\]
The above rank condition can be replaced by a set of equations, say,
$
\varphi_1(u) = \cdots = \varphi_N(u) = 0,
$
using the $2$-by-$2$ minors.
Hence, the critical values are feasible objective values of the problem
\[
\min \quad f(x) \quad s.t. \quad g(x) \geq 0, \,
\varphi_i(x) = 0 \,(i \in [N]).
\]
Then, the method in \S4
can be applied to get all critical values.
Properties like in Theorems~\ref{thm:min:c1} and \ref{thm:rth:cr}
also hold. For more general closed semialgebraic sets,
we refer to \cite{Nie-jac}
on how to construct polynomials $\varphi_i$.

\subsection{The number of local minimizers}

Every local minimizer is a critical point.
Hence, the number of local minimizers
is at most the number of critical points.
The number of all complex critical points, for generic polynomials,
is given in \cite{NR09}. However, not every critical point is a local minimizer.
Thus, such a number is generally an upper bound for the number of local minimizers,
but it is likely not sharp. For instance, in Example~\ref{exmp3.9},
the number of all complex critical points given by \cite{NR09}
is $64$, while there is only one local minimizer.
The number of local minimizers could be exponentially many.
For instance, for the polynomial
$\sum_{i=1}^n (x_i^2-1)^2$ (thanks to an anonymous referee), the number is $2^n$.

For a polynomial $f$, let $\mc{L}(f)$ be the set of its local minimizers.
Clearly, $\mc{L}(f)$ is a subset of $\mc{H}(f)$, the set of $H$-minimizers of $f$.
For generic $f$, these two sets are same.
The set $\mc{H}(f)$ is defined by $\nabla f(x) =0$, $\nabla^2 f(x) \succeq 0$.
So, $\mc{H}(f)$ is a basic closed semialgebraic set.
If $\mc{H}(f)$ is nonempty and bounded, then its number of connected
components is $b_0\big(\mc{H}(f)\big)$,
the zero-th Betti number of $\mc{H}(f)$. Therefore, the number of
local minimizers is generally equal to $b_0\big(\mc{H}(f)\big)$.
Usually, it is hard to compute $b_0\big(\mc{H}(f)\big)$.
We refer to \cite[\S6.2]{AlgRAG} for Betti numbers
of semialgebraic sets.

%
%

\subsection{Infimums of polynomials}

For a polynomial $f$, an important problem is to compute its infimum $f^*$
over the space $\re^n$. If $f^*$ is finite and achievable,
the gradient SOS method in \cite{NDS} is efficient for computing $f^*$.
If $f^*$ is finite but not achievable, that method cannot find $f^*$.
For such cases, there exist other type SOS relaxations for computing $f^*$
(cf.~\cite{GEZ10,Schw06,VuiSon}).
The problem about infimums
(i.e., decide whether $f^*$ is finite or not;
if finite, compute $f^*$ and decide whether $f^*$ is achievable or not)
can be solved by real algebraic techniques (e.g., quantifier eliminations).
We refer to the book~\cite{AlgRAG}
and the work~\cite{GrSED11,SED08} for such techniques.
It is an interesting future work to combine
real algebraic and semidefinite relaxation techniques
for computing infimums and the hierarchy of local minimums.

\bigskip \noindent
{\bf Acknowledgement} \,
The author would like to thank Didier Henrion
for his comments on using polynomial matrix inequalities
in {\tt GlotpiPoly} and {\tt YALMIP}. He also likes to thank
the editors and anonymous referees for the useful suggestions.
The research was partially supported by the NSF grants
DMS-0844775 and DMS-1417985.

\end{document}